%% file: BandSurgeriesAndCrossingChangesBetweenFiberedLinks.tex
\documentclass[11pt]{amsart}

\usepackage[utf8]{inputenc}

\usepackage{geometry}
\usepackage{graphicx}
\usepackage{amssymb}
\usepackage{ latexsym }
\usepackage{amsthm}

\usepackage{booktabs}
\usepackage{array} 
\usepackage{paralist} 
\usepackage{verbatim}
\usepackage{subfig} 

\usepackage{fancyhdr} 
\newcommand{\bd}{\partial}
\newcommand{\set}[1]{\{ #1 \}}
\newcommand{\rmv}{\smallsetminus}
\newcommand{\wt}[1]{\widetilde{#1}}
\newcommand{\wh}[1]{\widehat{#1}}
\theoremstyle{plain}
\newtheorem{lem}{Lemma}

\newtheorem{thm}{Theorem}

\newtheorem{cor}{Corollary}
\newtheorem*{convention}{Conventions}
\theoremstyle{remark}
\newtheorem*{example}{Example}
\newtheorem{rem}{Remark}
\theoremstyle{definition}
\newtheorem*{definition}{Definition}
\newtheorem*{acknowledgements}{Acknowledgements}

\title{Band~surgeries~and~crossing~changes between~fibered~links}

\author{Dorothy Buck}
\address{Department of Mathematics, Imperial College London, South Kensington, London, England SW7 2AZ}
\email{d.buck@imperial.ac.uk}
\urladdr{www2.imperial.ac.uk/~dbuck/}

\author{Kai Ishihara}
\address{Department of Mathematics and Information Sciences, Yamaguchi University, 1677-1 Yoshida, Yamaguchi-shi Yamaguchi 753-8513, Japan}
\email{kisihara@yamaguchi-u.ac.jp}
\author{Matt Rathbun}
\address{Department of Mathematics, Imperial College London, South Kensington, London, England SW7 2AZ}
\email{m.rathbun@imperial.ac.uk}
\urladdr{www2.imperial.ac.uk/~mrathbun/}

\author{Koya Shimokawa}
\address{Department of Mathematics, Saitama University, 255 Shimo-Okubo, Sakura-ku, Saitama, 380-8570, Japan}
\email{kshimoka@rimath.saitama-u.ac.jp}
\urladdr{http://www.rimath.saitama-u.ac.jp/lab.en/kshimoka/index-e.html}

\begin{document}
\maketitle
\begin{center}
\today
\end{center}

\begin{abstract}
We characterize cutting arcs on fiber surfaces that produce new fiber surfaces, and the changes in monodromy resulting from such cuts. As a corollary, we characterize band surgeries between fibered links and introduce an operation called Generalized Hopf banding. We further characterize generalized crossing changes between fibered links, and the resulting changes in monodromy.
\end{abstract}

\section{Introduction}

\input{introduction.tex}

\begin{acknowledgements}The authors would like to thank Tsuyoshi Kobayashi, Ryosuke Yamamoto, Alexander Coward, Ken Baker, and Miranda Antonelli for useful conversations, and the referee who helped shorten the length of Lemma 2 substantially.  We also thank the Isaac Newton Institute of Mathematical Sciences for hosting us during our early conversations.  DB is partially supported by EPSRC grants EP/H0313671, EP/G0395851 and  EP/J1075308, and gratefully acknowledges additional funding from the LMS and Leverhulme Trust.   MR is supported by EPSRC grant EP/G0395851, and KI was supported by EPSRC grant EP/H0313671.  
KS is supported by JSPS KAKENHI Grant Number 25400080.
\end{acknowledgements}

\section{Preliminaries}
\label{section:Definitions}
\input{preliminaries.tex}

\section{Cutting arcs in fiber surfaces}
\label{section:cuttingarcs}
\input{cuttingarcs.tex}

\section{Characterization of band surgeries on fibered links}
\label{section:bandsurgeries}
\input{band.tex}

\section{Dehn surgeries along arc-loops}
\label{section:loop}
\input{loop.tex}

\section{Characterization of generalized crossing changes between fibered links}
\label{section:crossingchanges}
\input{crossing.tex}

\section{An alternative proof of Theorem \ref{thm:whenfiber} }
\label{section:alternativeproof}
\input{alternativeproof.tex}

\bibliographystyle{plain} 
\bibliography{BandSurgeriesAndCrossingChangesBetweenFiberedLinks.bib}

\end{document}

%% file: introduction.tex
A \emph{fibered link} is one whose exterior fibers over $S^1$, so that each fiber is a Seifert surface for the link.
Among the many fascinating properties of fibered links is the ability to express their exteriors as a mapping torus, thereby
allowing us to encode the 3-dimensional information about the link exterior in terms of a surface automorphism. We will refer to this automorphism as the \emph{monodromy} of the link, or of the surface.
This connection yields generous amounts of information, including geometric classification (e.g. the link exterior is hyperbolic if and only if
the surface automorphism is pseudo-Anosov \cite{ThuGDDS}), topological
information (e.g. the fiber surface is the unique minimal genus Seifert
surface \cite{BurZieNKF}),
and methods to de-construct/re-construct fibered links \cite{HarHCAFKL}, \cite{StaCFKL}.

In addition to providing beautiful examples and visualizations of link exteriors, fibered links are deeply connected with important areas of topology, including the Berge Conjecture \cite{KirPLDT,OzsSzaKFHLSS,NiKFHDFK}, as well as contact geometry due to Giroux's correspondence \cite{GirCTC} between
open books and contact structures on 3-manifolds. 

In this paper we further explore constructions of fibered links in terms of the monodromy. We will generalize a very well-understood and important operation on fiber surfaces known as Hopf plumbing (or its inverse, Hopf de-plumbing): 
If a fiber surface has a Hopf plumbing summand,
then
cutting along the spanning arc of the Hopf annulus
results in another fiber surface,
and this process is called Hopf de-plumbing.
It is known, for instance, that any fiber surface of a fibered link in $S^3$ can be constructed from a disk by a sequence of Hopf plumbings and Hopf de-plumbings \cite{GirGooOSEOB3M}.
Such an arc corresponding to a Hopf plumbing can be characterized in terms of the monodromy.
We will characterize all arcs on a fiber surface cutting along
which gives another fiber surface.
This will lead naturally to the construction of a \emph{generalized Hopf banding},
and we will leverage our results to relate to 
two other crucially important operations: band surgeries, and generalized crossing changes.
We will complete the characterization of band surgeries between fibered links,
and (generalized) crossing changes between fibered links.

The paper is organized as follows: In Section \ref{section:Definitions}, we provide definitions and background for the tools we will use.
In Section \ref{section:cuttingarcs} we study the result of cutting a fiber along an arc, and prove:

\begin{thm} \label{thm:whenfiber} 
Let $L$ be 
a fibered (oriented) link
with 
fiber $F$, 
monodromy $h$ (which is assumed to be the identity on $\bd F$), and suppose $\alpha$ is a properly embedded arc in $F$.
Let $F'$ be the surface obtained by cutting $F$ along $\alpha$, and the resulting (oriented) link $L' = \bd F'$.
The surface $F'$ is a fiber for
$L'$ if and only if $i_{total}(\alpha) = 1$ (that is, when $\alpha$ is clean and alternating, or once-unclean and non-alternating).
\end{thm}

 \noindent We also characterize the resulting changes in monodromy, see Corollary \ref{cor:monodromy-band}. (See Section \ref{section:Definitions} for the definition of $i_{total}(\alpha)$.)
\smallskip

By \cite{SchaThoLGCM,HirShiDSSIKYLS,DerMiyMotNSSKIII}, it is known that if a \emph{coherent band surgery} increases the Euler
characteristic of a link,
then the band can be isotoped onto a taut Seifert surface.
Hence, such a band surgery between fibered links
corresponds to cutting the fiber surface.
When a coherent band surgery changes the Euler characteristic of a link
by at least two,
such a band surgery is characterized by
Kobayashi \cite{KobFLUO}
(see Theorems \ref{thm:4equiv} and \ref{thm:Kobayashi-band}).
In Section \ref{section:bandsurgeries}, we introduce \emph{generalized Hopf
banding} and give a characterization of the remaining case.

\begin{thm}\label{thm:fiberedlinks}
Suppose $L$ and $L'$ are links in $S^3$, and $L'$ is obtained from $L$ by a coherent band surgery
and $\chi(L')=\chi(L)+1$.

{\rm (1)}  Suppose $L$ is a fibered link. Then $L'$ is a fibered link if and only if 
the fiber $F$ for $L$ is a generalized Hopf banding of a Seifert surface $F'$ for $L'$ along $b$. 

{\rm (2)} Suppose $L'$ is a fibered link. Then $L$ is a fibered link if and only if 
a Seifert surface $F$ for $L$ is a generalized Hopf banding of the fiber $F'$ for $L'$ along $b$. 
\end{thm}

As an application, we characterize band surgeries on torus links $T(2,p)$,
or connected sums of those, that produce fibered links
(Corollaries \ref{cor:2torus} and \ref{cor:consumtorus}). In the forthcoming paper \cite{BucIshRatShiII}, 
we use this to completely characterize an important biological operation: the unlinking of DNA molecules by a recombinase system.

In Section \ref{section:loop}, we consider \emph{arc-loops}, which are loops around arcs, 
and characterize Dehn surgeries along arc-loops preserving a fiber surface, 
using results of Ni \cite{NiDSKPM} about surgeries on knots in trivial sutured manifolds.
\begin{thm}\label{thm:surgery}
Suppose $F$ is a fiber surface in a $3$-manifold $M$ and $c$ is an $\alpha$-loop $($a loop around an arc $\alpha)$. 
Suppose that $\gamma$ is a non-trivial slope on $c$, and that $N(\gamma)$ is the manifold obtained from $M$ via the $\gamma$-surgery on $c$. 
Then $F$ is a fiber surface in $N(\gamma)$ if and only if 
\begin{enumerate}[\rm (1)]
\item $\alpha$ is clean and $\gamma=i_{\partial}(\alpha)+\frac{1}{n}$  for some integer $n$, or
\item $\alpha$ is once-unclean and $\gamma=i_{\partial}(\alpha)$.
\end{enumerate}
\end{thm}
\noindent (See Section \ref{section:Definitions} for the definition of $i_{\bd}(\alpha)$.)
\smallskip

By \cite{SchaThoLGCM, KobFLUO}, crossing changes between fibered links with different Euler characteristics are understood. In Section \ref{section:crossingchanges}, we investigate the remaining case and characterize when fibered links of the same Euler characteristic are related by crossing changes and generalized crossing changes.

\begin{thm}\label{thm:crossingchange} 
Suppose a link $L'$ is obtained from a fibered link $L$ in $S^3$ 
with fiber $F$ by a crossing change, and $\chi(L')=\chi(L)$. 
Then $L'$ is a fibered link if and only if the crossing change is a Stallings twist or an $\varepsilon$-twist along an arc $\alpha$ in $F$, 
where $\alpha$ is once-unclean and alternating with $i_{\bd}(\alpha)=-\varepsilon$.
\end{thm}

\begin{thm}\label{thm:gencrossingchange}
Suppose $L$ and $L'$ are fibered links in $S^3$
related by a generalized crossing change with $\chi(L) = \chi(L')$.
Then the generalized crossing change is an $n$-twist around an arc $\alpha$, and one of the following holds:
\begin{enumerate}[\rm (1)]
\item $\alpha$ is clean and non-alternating,  
\item $n=\pm2$, and $\alpha$ is clean and alternating with $i_{\partial}(\alpha)=-n/2$, or
\item $n=\pm1$, and $\alpha$ is once-unclean and alternating with $i_{\partial}(\alpha)=-n$.
\end{enumerate}
\end{thm}

\noindent The generalized crossing change of (1) in Theorem \ref{thm:gencrossingchange} implies a Stallings twist of type $(0,1)$ (see Theorem \ref{thm:4equiv}). For the generalized crossing change of (2) in Theorem \ref{thm:gencrossingchange}, the  crossing circle links a plumbed Hopf annulus and the $\pm2$-twist reverses the direction of twist in the Hopf annulus (see Theorem \ref{thm:Hopfband}). The resulting changes in monodromy for (3) of Theorem \ref{thm:gencrossingchange} are characterized in Corollary \ref{cor:gencrossingchange-monodromy}.
\smallskip

In Section \ref{section:alternativeproof}, we give an alternative proof of Theorem \ref{thm:whenfiber} using Theorem \ref{thm:surgery}.

%% file: preliminaries.tex
\subsection{Surfaces}

\begin{definition} 
A Seifert surface $F$ for a link $L$ is \emph{taut} if it maximizes Euler characteristic over all Seifert surfaces for $L$. We say the 
Euler characteristic for the link is $\chi(L) = \chi(F)$ if $F$ is taut.
\end{definition}

\begin{definition}[see \cite{GooOOBSA}]
Let $\alpha$, $\beta$ be two oriented arcs properly embedded in an oriented surface $F$ which intersect transversely. At a point $p \in \alpha \cap \beta$, define $i_p$ to be $\pm 1$ depending on whether the orientation of the tangent vectors $(T_p \alpha, T_p \beta)$ agrees with the orientation of $F$ or not. If $\alpha$ and $\beta$ intersect minimally over all isotopies fixing the boundary pointwise, then the following are well-defined:
\begin{enumerate}[\rm (1)]
\item The \emph{geometric intersection number}, $\rho(\alpha, \beta) := \sum_{p \in \alpha \cap \beta \cap int(F)} |i_p|$, is the number of intersections (without sign) between $\alpha$ and $\beta$ in the interior of $F$. 
\item The \emph{boundary intersection number}, $i_\bd (\alpha, \beta) := \frac{1}{2} \sum_{p \in \alpha \cap \beta \cap \bd F} i_p$, is half the sum of the oriented intersections at the boundaries of the arcs. 
\item The \emph{total intersection number}, $i_{total}(\alpha, \beta) := \rho(\alpha, \beta) + |i_\bd (\alpha, \beta)|$, is the sum of the (unoriented) interior intersections between $\alpha$ and $\beta$, and the absolute value of half the sum of the boundary intersections between the two arcs.
\item If $F$ is 
a fiber surface 
with monodromy $h$, then we define $\rho(\alpha) := \rho(\alpha, h(\alpha))$, and $i_\bd(\alpha) := i_\bd (\alpha, h(\alpha))$, and $i_{total} (\alpha) := i_{total}(\alpha, h(\alpha))$.
\end{enumerate}  
\end{definition}

\begin{definition}[see \cite{GabDFLS3}]
\label{def:Murasugisum}
Let $F_i \subset M_i$, for $i = 1, 2$, be compact oriented surfaces in the closed, oriented 3-manifolds $M_i$. Then $F \subset M_1 \# M_2 = M$ is a \emph{Murasugi sum} of $F_1$ and $F_2$ if
$$M = (M_1 \rmv int(B_1)) \cup_{S^2} (M_2 \rmv int(B_2)), \, \, \, \mbox{for 3-balls } B_i \mbox{ with } S^2 = \bd B_1 = \bd B_2,$$
and for each $i$,
$$S^2 \cap F_i \mbox{ is a  2}n\mbox{-}gon, \, \, \, \, \, \mbox{ and }\, \, \, \, \, (M_i \rmv int(B_i)) \cap F = F_i.$$
When $n=2$, this is known as a \emph{plumbing} of $F_1$ and $F_2$. Further, when $n=2$ and one of the surfaces, say $F_2$ is a Hopf annulus, this is known as a \emph{Hopf plumbing}. 
\end{definition}

\subsection{Sutured manifolds}

\begin{definition}[see \cite{GabFT3M,SchaSMGTN,GodHSSMMS}]
A \emph{sutured manifold}, $(N, \gamma)$, is a compact 3-manifold $N$, with a set $\gamma \subset \bd N$ of mutually disjoint annuli, $A(\gamma)$, and tori, $T(\gamma)$, satisfying the orientation conditions below. (We will only consider the case when $T(\gamma) = \emptyset$.) Call the core curves of the annuli $A(\gamma)$ the \emph{sutures}, and denote them $s(\gamma)$. Let $R(\gamma) = \overline{\bd N \rmv A(\gamma)}$. 

\begin{enumerate}
\item Every component of $R(\gamma)$ is oriented, and $R_+(\gamma)$ (respectively, $R_-(\gamma)$) denotes the union of the components whose normal vectors point out of (resp., into) $N$.
\item The orientations of $R_{\pm}(\gamma)$ are consistent with the orientations of $s(\gamma)$.
\end{enumerate}

\end{definition}

We will often simplify notation and write $(N, s(\gamma))$ in place of $(N, \gamma)$.

\begin{definition}[see \cite{GabFT3M,SchaSMGTN,GodHSSMMS}]
 We say that a sutured manifold $(N, \gamma)$ is a \emph{trivial sutured manifold} if it is homeomorphic to $(F \times I, \bd F \times I)$, for some compact, bounded surface $F$, with $R_+(\gamma) = F \times \set{1}$, $R_-(\gamma) = F \times \set{0}$, and $A(\gamma) = \bd F \times I$.
\end{definition}

\begin{definition}[see \cite{GabFT3M,SchaSMGTN,GodHSSMMS}]
Suppose $F$ is a Seifert surface for an oriented link $L$ in a manifold $M$. Then $(n(F), L) = (F \times I, \bd F \times \set{\frac{1}{2}})$ is a trivial sutured manifold. We call $(\overline{M \rmv n(F)}, L)$ the \emph{complimentary sutured manifold}.
\end{definition}

\begin{definition}[see \cite{GabFT3M,SchaSMGTN,GodHSSMMS}]
A properly embedded disk $D$ in $(N, \gamma)$ is a \emph{product disk} if $\bd D \cap A(\gamma)$ consists of two essential arcs in $A(\gamma)$. A \emph{product decomposition} is an operation to obtain a new sutured manifold $(N', \gamma')$ from a sutured manifold $(N, \gamma)$ by decomposing along an oriented product disk (see \cite{GabFT3M}). We denote this
$$ (N, \gamma) \stackrel{D}{\leadsto} (N', \gamma').$$  
\end{definition}

The following definition and theorem are due to Wu \cite{WuGHAT}. (Wu's definition is slightly more general, but we will only need the special case described here.)

\begin{definition}[\cite{WuGHAT}]
Let $M$ be a 3-manifold with non-empty boundary, and let $\gamma$ be a collection of essential simple closed curves in $\bd M$. An $n$-compressing disk (with respect to $\gamma$) is a compressing disk for $\bd M$ which intersects $\gamma$ in $n$ points; also call a compressing disk for $\bd M \rmv \gamma$ a 0-compressing disk.
\end{definition}

\begin{thm}[\cite{WuGHAT}]
\label{thm:ncompressingdisks}
Let $M$ be a 3-manifold with non-empty boundary, let $\gamma$ be a collection of essential simple closed curves in $\bd M$, and let $J$ be a simple closed curve in $\bd M$ disjoint from $\gamma$. Suppose that $\bd M \rmv \gamma$ is compressible. Let $M'$ be the result of attaching a 2-handle to $M$ along $J$. If $\bd M'$ is $n$-compressible, then $\bd M \rmv J$ is $k$-compressible for some $k \leq n$.
\end{thm}

\begin{convention} Let us establish some informal conventions to aid in visualization:
\begin{enumerate}[\rm (1)]
\item $F \times [0, 1]$ will refer to a product where the $[0, 1]$ component is `vertical', with $F \times \set{0}$ on the `top', and $F \times \set{1}$ on the `bottom'.
\item Correspondingly, the orientation of $F$ will be such that $F \times \set{0}$ corresponds to the `positive' side of $F$.
\item $[0, 1] \times D$ will refer to a product where the $[0, 1]$ component is `horizontal', with $\set{0} \times D$ on the `left', and $\set{1} \times D$ on the `right'.
\item A fibered link complement will be thought of as arising from a mapping torus $(F \times [0, 1]) / h$, where $h: F \times \set{1} \to F \times \set{0}$, so that the product disk determined by $\alpha$ and $h(\alpha)$ will emanate `downwards' from $\alpha$ in $F \times \set{1}$, and `upwards' from $h(\alpha)$ in $F \times \set{0}$.
\end{enumerate}
\end{convention}

%% file: cuttingarcs.tex
In this section we will give a direct proof of Theorem \ref{thm:whenfiber}. 
See Section \ref{section:alternativeproof} for an alternative proof using Ni's result \cite{NiDSKPM}.
Let $L$ be 
a fibered (oriented) link in a manifold $M$
with 
fiber $F$, 
monodromy $h$ (which is assumed to be the identity on $\bd F$), and suppose $\alpha$ is a properly embedded arc in $F$. Assume $\alpha$ and $h(\alpha)$ have been isotoped in $F$, fixing the endpoints, to intersect minimally. If the endpoints of $h(\alpha)$ emanate to opposite sides of $\alpha$, then $|i_\bd(\alpha)| = 1$. In this case, $\alpha$ is called \emph{alternating}. Otherwise, $|i_\bd(\alpha)| = 0$, and $\alpha$ is called \emph{non-alternating}. If $\rho(\alpha) = 0$, then $\alpha$ is said to be \emph{clean}. If $\rho(\alpha)=n > 0$, $\alpha$ is said to be \emph{$n$-unclean}. (See Figure \ref{fig:algint}.)
 \begin{figure}[h]
\begin{center}
\begin{minipage}{.23\textwidth}%%%clean&non-alternating
\begin{center}
\begin{picture}(80,40)(0,0)
\put(33,25){$\alpha$}
\linethickness{0.3mm}
\qbezier[20](30,0)(30,20)(30,40)
\put(0,20){$h(\alpha)$}
\linethickness{0.3mm}
\qbezier(30,0)(15,10)(0,10)
\qbezier(30,40)(15,30)(0,30)
\linethickness{0.5mm}
\put(0,0){\line(1,0){60}}`
\put(0,40){\line(1,0){60}}
\end{picture}
clean and\\non-alternating
\end{center}
\end{minipage}%%%
\begin{minipage}{.23\textwidth}%%%clean&alternating
\begin{center}
\begin{picture}(80,40)(0,0)
\put(33,25){$\alpha$}
\linethickness{0.3mm}
\qbezier[20](30,0)(30,20)(30,40)
\put(0,20){$h(\alpha)$}
\linethickness{0.3mm}
\qbezier(30,0)(45,10)(60,10)
\qbezier(30,40)(15,30)(0,30)
\linethickness{0.5mm}
\put(0,0){\line(1,0){60}}
\put(0,40){\line(1,0){60}}
\end{picture}
clean and\\alternating
\end{center}
\end{minipage}
\begin{minipage}{.23\textwidth}
\begin{center}
\begin{picture}(80,40)(0,0)
\put(33,25){$\alpha$}
\linethickness{0.3mm}
\qbezier[20](30,0)(30,20)(30,40)
\put(0,20){$h(\alpha)$}
\linethickness{0.3mm}
\qbezier(30,0)(15,10)(0,10)
\qbezier(30,40)(15,30)(0,30)
\qbezier(0,15)(30,15)(60,15)
\linethickness{0.5mm}
\put(0,0){\line(1,0){60}}
\put(0,40){\line(1,0){60}}
\end{picture}
once-unclean and\\non-alternating
\end{center}
\end{minipage}
\begin{minipage}{.23\textwidth}
\begin{center}
\begin{picture}(80,40)(0,0)
\put(33,25){$\alpha$}
\linethickness{0.3mm}
\qbezier[20](30,0)(30,20)(30,40)
\put(0,20){$h(\alpha)$}
\linethickness{0.3mm}
\qbezier(30,0)(45,10)(60,10)
\qbezier(30,40)(15,30)(0,30)
\qbezier(0,15)(30,15)(60,15)
\linethickness{0.5mm}
\put(0,0){\line(1,0){60}}
\put(0,40){\line(1,0){60}}
\end{picture}
once-unclean and\\alternating
\end{center}
\end{minipage}
\caption{}
\label{fig:algint}
\end{center}
\end{figure}

\begin{rem} If the arc $\alpha$ is fixed by the monodromy, then $h(\alpha)$ can be isotoped to have interior disjoint from $\alpha$, so this is a special case of a clean, non-alternating arc.
\end{rem}

Let $F'$ be the surface obtained by cutting $F$ along $\alpha$, and call the resulting (oriented) link $L' = \bd F'$.
We now restate Theorem \ref{thm:whenfiber}.
\medskip

\noindent
{\bf Thereom \ref{thm:whenfiber}.}\ 
{\it The surface $F'$ is a fiber for
$L'$ if and only if $i_{total}(\alpha) = 1$ (that is, when $\alpha$ is clean and alternating, or once-unclean and non-alternating).}
\medskip

Consider the 
fiber 
$F$, and a small product neighborhood $n(F) = F \times I$. This is a trivial sutured manifold, $(n(F), \bd F)$. Let $D_-$ be the product disk $\alpha \times I$. Now, because 
$F$ is a fiber for $L$, 
the complementary sutured manifold 
$(\overline{M \rmv n(F)}, \bd F)$ 
is also trivial. Let $D_+$ be the product disk determined by $\alpha
\subset F \times \set{1}$ and $h(\alpha) \subset F \times \set{0}$,
properly embedded in 
$\overline{M \rmv n(F)}$.
As $D_-$ is a product disk for $(n(F), \bd F)$, we may decompose along this disk to get another trivial sutured manifold, namely $(n(F'), \bd F')$. 
 
Recall that the manifold $n(F')$ was obtained by removing a small product neighborhood of $D_-$, say $[0, 1] \times D_- = [0, 1] \times (\alpha \times [0, 1])$, from $n(F)$. Let $B$ be the ball $[-1, 2] \times (\alpha \times [-1, 1])$. Now, attach to $(n(F'), \bd F')$ the $1$-handle $([-1, 2] \times (\alpha \times [-1, 0]))$, (attached along $([-1, 0] \times (\alpha \times \set{0}))$ and $([1, 2] \times (\alpha \times \set{0}))$). Call the resulting sutured manifold $(N_1, \bd F')$ (see Figure \ref{fig:N1}). We will refer to $(F' \times \set{1}) \subset \bd N_1$ as $\bd_- N_1$, and $\overline{\bd N_1 \rmv ((F' \times \set{1}) \cup (\bd F' \times I))}$ as $\bd_+ N_1$.

\begin{figure}
\begin{center}
\includegraphics[width=4in]{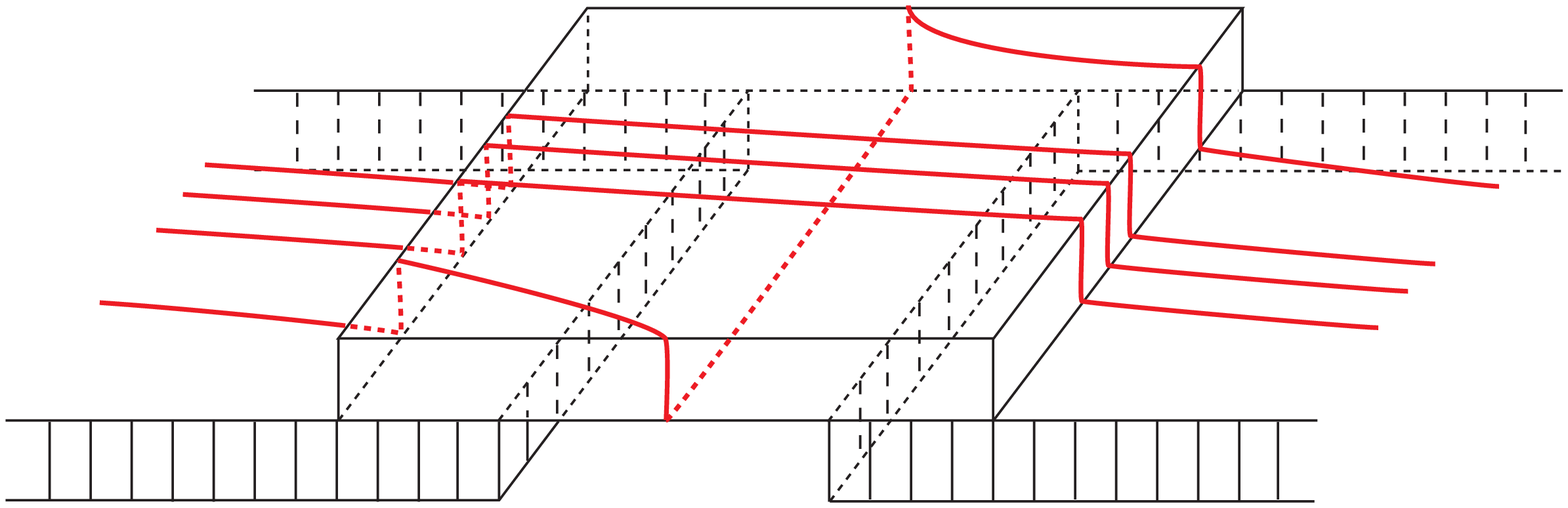}

\begin{picture}(400,0)(0,0)
\put(168,7){$\partial D'_{+}$}
\put(175,20){$\uparrow$}
\end{picture}
\caption{The sutured manifold $(N_1, \bd F')$. The boundary of the modified disk $\bd D'_+$ reflects the pattern of the product disk $D_+$.}
\label{fig:N1}
\end{center}
\end{figure}

Now, we can modify $D_+$ to a new disk $D'_+$ in 
$\overline{M \rmv N_1}$ 
as follows:
\begin{enumerate}
\item Let $\widetilde{D_+} = \overline{D_+ \rmv B}$. Note that $h(\alpha)$ corresponds, through vertical projection in $B$, to arcs in $(\set{-1} \times (\alpha \times [-1, 0])) \cup ([-1, 2] \times (\alpha \times \set{-1})) \cup (\set{2} \times (\alpha \times [-1, 0]))$.
\item Extend the sub-arc $\alpha \times \set{1}$ of $\widetilde{D_+}$ through vertical projection in $B$ to $\alpha \times \set{0}$.
\end{enumerate}

\begin{lem}\label{lem:uniquecompressingdiskandproductdisks} The co-core of the $1$-handle is the unique %compressing 
non-separating disk in $N_1$ disjoint from the sutures. Furthermore, every product disk in $(N_1, \bd F')$ can be made disjoint from this disk.
\end{lem}

\begin{proof}
Let $D$ be the co-core of the $1$-handle, and suppose $D'$ is either $(A)$ a distinct non-separating compressing disk disjoint from $\bd F'$ or $(B)$ a product disk for the sutured manifold; in either case chosen so as to minimize $|D \cap D'|$. Suppose $l$ were a loop of intersection innermost in $D$. Then $l$ bounds a subdisk $\widetilde{D}$ of $D$ and a subdisk $\widetilde{D'}$ of $D'$. These two disks co-bound a sphere, which then bounds a $3$-ball because $N_1$ is irreducible. This sphere provides a means of isotoping $\widetilde{D'}$ to $\widetilde{D}$, which reduces the number of loops in $D \cap D'$, contradicting minimality. Thus, we may suppose that $D \cap D'$ consists only of arcs. In this case, an outermost arc in $D$ provides a boundary compression of $D'$, separating $D'$ into two disks. At least one of these disks is a disk of the same type as $D'$ (i.e., $(A)$ or $(B)$ above), but which intersects $D$ fewer times than $D'$, again contradicting minimality. 

So we may assume now that $D' \cap D = \emptyset$. Hence, we may isotope $D'$ completely off of the $1$-handle. In case $(B)$, this establishes the second statement. In case $(A)$, either $D'$ is a compressing disk for $F' \times \set{0}$ in $F' \times I$, or it is boundary parallel to a disk in $F' \times \set{0}$. Neither of these is possible since $D'$ is non-separating, which establishes the first statement. 
\end{proof}

Now, we attach a $2$-handle to $N_1$ along a neighborhood of $\bd
D'_+$. Call the resulting sutured manifold $(N_2, \bd F')$, and keep
track of the attaching annulus, $A = n(\bd D'_+)$, on the one hand
thought of as contained in the boundary of $N_1$, and on the other hand
considered to be properly embedded in $N_2$. Observe that since $D_+$
was a product disk for the trivial sutured manifold 
$(\overline{M \rmv n(F)}, \bd F)$ 
which is homeomorphic to $N_1$, attaching the $2$-handle along $\bd
D'_+$ results in the same manifold as decomposing 
$(\overline{M \rmv n(F)}, \bd F)$ 
along $D_+$. Furthermore, the sutures in 
$(\overline{M \rmv N_2}, \bd F')$ 
can be slid along $D'_+$, and can be seen to agree with the result of
this decomposition. Therefore, as a sutured manifold, 
$(\overline{M \rmv N_2}, \bd F')$ 
is the same as the result of decomposing 
$(\overline{M \rmv n(F)}, \bd F)$ 
along the product disk $D_+$, and is thus a trivial sutured manifold. 

We conclude then that $F'$ is a 
fiber 
in 
a fibration 
for 
$L'$ if and only if $(N_2, \bd F')$ is trivial.

\begin{rem} We remind the reader that $(N_1, \bd F')$ is simply constructed from $(n(F'), \bd F')$ by attaching a $1$-handle along $F' \times \set{0}$, and that $(N_2, \bd F')$ is constructed by attaching a $2$-handle to $(N_1, \bd F')$.
\end{rem}

\begin{lem}\label{lem:cancelingdisk} If $(N_2, \bd F')$ is trivial, then there exists a compressing disk for $(N_1, \bd F')$, disjoint from the sutures, which intersects the boundary of $D'_+$ exactly once.
\end{lem}

\begin{proof}

Begin with a maximal collection $\mathcal{D}$ of product disks for $(N_1, \bd F')$ that are disjoint from $\bd D'_+$. By Lemma \ref{lem:uniquecompressingdiskandproductdisks}, these disks can also be taken disjoint from the 1-handle of $N_1$. Since these will also be product disks for $(N_2, \bd F')$, the result of attaching the 2-handle along $\bd D'_+$ after this decomposition will be trivial if and only if $(N_2, \bd F')$ is trivial. Further, the result of decomposing $(N_1, \bd F')$ along the collection $\mathcal{D}$ is still a surface cross an interval, with a 1-handle attached. Thus, without loss of generality, we may assume the $(N_1, \bd F')$ has no product disks disjoint from $\bd D'_+$.

If $F'$ were a disk, then $(N_1, \bd F')$ would be a solid torus. In this case, $\bd D'_+$ must intersect the co-core of the 1-handle exactly once, else $(N_2, \bd F')$ would be a punctured lens space, and not a trivial sutured manifold.

Let us then assume that $F'$ is not a disk. It follows that there is an essential product disk in $(N_2, \bd F')$. A product disk intersects the sutures in two points, so by Theorem \ref{thm:ncompressingdisks}, there must be a compressing disk $D$ in $(N_1, \bd F')$ with boundary disjoint from $\bd D'_+$, and intersecting the sutures in at most two points. The disk $D$ cannot intersect the sutures in two points, or else it would be a product disk for $(N_1, \bd F')$ disjoint from $\bd D'_+$, contrary to the maximality condition of the initial collection of product disks. Further, it is not possible that $D$ intersect the sutures in just one point, since the sutures are separating in $\bd N_2$. Thus, $D$ is a compressing disk disjoint from $\bd D'_+$ and the sutures.

If $D$ were non-separating, then Lemma \ref{lem:uniquecompressingdiskandproductdisks} says that $D$ would be the co-core of the 1-handle. But then, $\bd D'_+$, being disjoint from $D$, could be isotoped completely off of the 1-handle, so that $D'_+$ would be a compressing disk for $F' \times \set{0}$ in $\overline{M \rmv n(F')}$, which cannot occur in a trivial sutured manifold.

Hence, $D$ is a separating disk. By an outermost arc argument similar to the proof of Lemma \ref{lem:uniquecompressingdiskandproductdisks}, we can show that $D$ may be assumed to be disjoint from the co-core of the 1-handle. Since $D$ cannot be a compressing disk for $F' \times I$, it must be that $\bd D$ is essential in $\bd_+ N_1$, but not in $F' \times \set{0}$. Hence, $D$ is parallel to a disk $D_{\bd}$ in $F' \times \set{0}$, which must contain both feet of the 1-handle. 

In this case, the region between $D$ and $D_\bd$, together with the 1-handle again forms a solid torus, and $\bd D'_+$ must intersect the co-core of the 1-handle exactly once, lest there be a punctured lens space in a trivial sutured manifold.

\end{proof}

\begin{proof}[Proof of Theorem \ref{thm:whenfiber}]
Combining the results of Lemmas \ref{lem:uniquecompressingdiskandproductdisks} and \ref{lem:cancelingdisk}, we see that if $F'$ is a 
fiber for $L',$
then $|\bd D \cap \bd D'_+| = 1$. Since $\bd D'_+$ reflects the product disk $D_+$, and therefore the pattern of $\alpha$ and $h(\alpha)$ on $F$, this shows that $\alpha$ must be either alternating and clean, or non-alternating and once-unclean.

Conversely, we know that if $\alpha$ were alternating and clean, then $F'$ would be 
the fiber of a fibration for $L'.$
Thus, it remains to show that if $\alpha$ is non-alternating and once-unclean, then $(N_2, \bd F')$ is trivial. This is shown by observing that in this case, $\bd D$ and $\bd D'_+$ form a canceling pair. The sutured manifold $(N_2, \bd F')$ is the result of attaching the $1$-handle with co-core $D$ to $(F' \times I, \bd F')$, and then a $2$-handle along $\bd D'_+$. As these are canceling handles, disjoint from the sutures, this is equivalent to doing neither, so that $(N_2, \bd F') \cong (F' \times I, \bd F')$, which is clearly a product sutured manifold. This completes the proof of Theorem \ref{thm:whenfiber}.
\end{proof}

\begin{rem} Observe that if $F'$ is not a fiber surface, this does not necessarily imply that $L'$ is not 
fibered.
It is possible that 
$L'$ fibers with a different surface as a fiber. 
We combine our results with those of Kobayashi to address this question when the manifold $M$ is a rational homology $3$-sphere in Section \ref{section:bandsurgeries}.
 \end{rem}

%% file: band.tex
In this section we will characterize band surgeries.
Throughout this section, $L$ and $L'$ are oriented links in a manifold $M$  
related by a \emph{coherent band surgery} along a band $b$.
More precisely,  $b$ is an embedding $[0,1]\times [0,1]\to M$ 
such that 
$b^{-1}(L)=[0,1]\times\{0,1\} $, $b^{-1}(L')=\{0,1\}\times [0,1]$, 
and $L$ and $L'$ are the same as oriented sets except on $b([0,1]\times[0,1])$. 
For simplicity, we use the same symbol $b$ to denote the image $b([0,1]\times[0,1])$.
Since the numbers of components of $L$ and $L'$  differ by $1$,  
their Euler characteristics will never be equal.  By \cite{SchaThoLGCM,HirShiDSSIKYLS,DerMiyMotNSSKIII}, there exists a taut Seifert surface $F$ for $L$ or $L'$, say $L$, 
such that $F$  contains $b$, and so $\chi(L')>\chi(L)$. 
\begin{thm}[\cite{SchaThoLGCM,HirShiDSSIKYLS,DerMiyMotNSSKIII}]\label{thm:Euler}
Suppose $L$ and $L'$ are links in $S^3$.  
Then $\chi(L')>\chi(L)$ if and only if $L$ has a taut Seifert surface $F$ containing $b$. 
\end{thm}
Suppose $L$ is a fibered link.
Then $F$ is a fiber surface for $L$, and the band $b$ is contained in $F$.
Call $\alpha:=b(\{\frac{1}{2}\}\times[0,1])$ the \emph{spanning arc} of the band. 
The surface $F'$, which is obtained by cutting $F$ 
along $\alpha$, can be regarded as  a Seifert surface for $L'$ by moving $F'$ slightly along $b$. 
Note that $\alpha$ is fixed by the monodromy of $F$ if  and only if $F'$ is a split union of two fiber surfaces, {\it i.e.} $L$ is a connected sum of the components of the split link $L'$.
Kobayashi characterized band surgeries in the case that $\chi(L')>\chi(L)+1$. 
By Kobayashi \cite{KobFLUO} and Yamamoto \cite{YamSTWCBRPDHB}, we have the following:
\begin{thm}
\label{thm:4equiv}
Suppose $L$ is a fibered  link in $S^3$. 
Then the following conditions are equivalent. 
\begin{enumerate}[\rm(1)]
\item $\chi(L')>\chi(L)+1$.
\item $F'$ is a pre-fiber surface.
\item There exists a disk $D$ such that 
the intersection of $D$ and $F$ is a disjoint union of $\partial D$ and $\alpha$, and $\partial D$ is essential in $F$ (hence Stallings twists of type $(0,1)$ can be performed).
\item $\alpha$ is clean and non-alternating but not fixed by the monodromy.
\end{enumerate}
\end{thm}
See \cite{KobFLUO} for the definition of pre-fiber surfaces. 
Moreover Kobayashi showed the following:  
\begin{thm}[\cite{KobFLUO}]\label{thm:Kobayashi-band}
Suppose $F$ is a fiber surface and $F'$ is pre-fiber surface, then 
the band $b$ is $\lq\lq$type {\rm F}''  with respect to $F'$.
\end{thm}
See \cite{KobFLBCSTL} for the definition of type F. 
He also characterized pre-fiber surfaces for fibered links in \cite[Theorem 3]{KobFLUO} and for split links in \cite[Theorem 3]{KobFLBCSTL}. 
In particular, together with Theorem \ref{thm:4equiv} and \cite[Theorem 3]{KobFLUO}, 
Theorem \ref{thm:Kobayashi-band} gives a complete characterization of
band surgeries between fibered links $L$ and $L'$ with
$\chi(L')>\chi(L)+1$. 
\bigskip
\begin{proof}[Proof of Theorem \ref{thm:4equiv}]
~\\
(1)$\Rightarrow$(2) See Kobayashi \cite[Theorem 2.1]{KobFLUO}.
\\
(2)$\Rightarrow$(1) If $F'$ is a pre-fiber surface, then, by  the definition of pre-fiber surfaces, 
it is compressible, and so $\chi(L')>\chi(F')=\chi(F)+1=\chi(L)+1$.
\\
(2)$\Rightarrow$(3) This follows from Theorem \ref{thm:Kobayashi-band} and the definition of type F.
\\
(3)$\Leftrightarrow$(4) See Yamamoto \cite[Lemma 3.4]{YamSTWCBRPDHB}.
\\
(4)$\Rightarrow$(2) See Kobayashi \cite[Proposition 4.5]{KobFLUO}.
\end{proof}
For the remaining case, we will characterize band surgeries 
between $L$ and $L'$ with $\chi(L')=\chi(L)+1$. 
In this case, $F$ is a fiber surface. 
By Theorem \ref{thm:whenfiber}, then $F'$ is a fiber surface for $L'$
if and only if  $\alpha$ is  clean and alternating, or once-unclean and non-alternating.
We will translate these conditions of the arc $\alpha$ into conditions of the band $b$.

\subsection{Hopf banding and generalized Hopf banding}
First we show that if the spanning arc of a band surgery is a clean alternating arc, then the band surgery corresponds to a Hopf plumbing.
If $F$ is obtained by plumbing of 
a surface $F''$ and a Hopf annulus $A$, 
then $F$ is obtained by attaching a band 
$\overline{A\rmv F''}$ to $F''$,
and so we call $F$ a {\em Hopf banding} of 
$F''$ along 
$\overline{A\rmv F''}$.
While known for some time, proofs of the following theorem can be found in Sakuma \cite{SakMDCSSUO}, or Coward and Lackenby \cite[Theorem 2.3]{CowLacUGOK}:
\begin{thm}\label{thm:Hopfband}
Suppose $F$ is a fiber surface.  
Then $F$ is a Hopf banding of $F'$ if and only if $\alpha$ is clean and alternating.
More precisely, the Hopf annulus of  the Hopf banding has a right-handed twist or left-handed twist,  depending on whether $i_{\partial}(\alpha)=1$ or $i_{\partial}(\alpha)=-1$ $($see Figure \ref{fig:Hopfannuli}$)$. 
\end{thm}
\begin{figure}[h]
\begin{center}
\includegraphics[width=10cm]{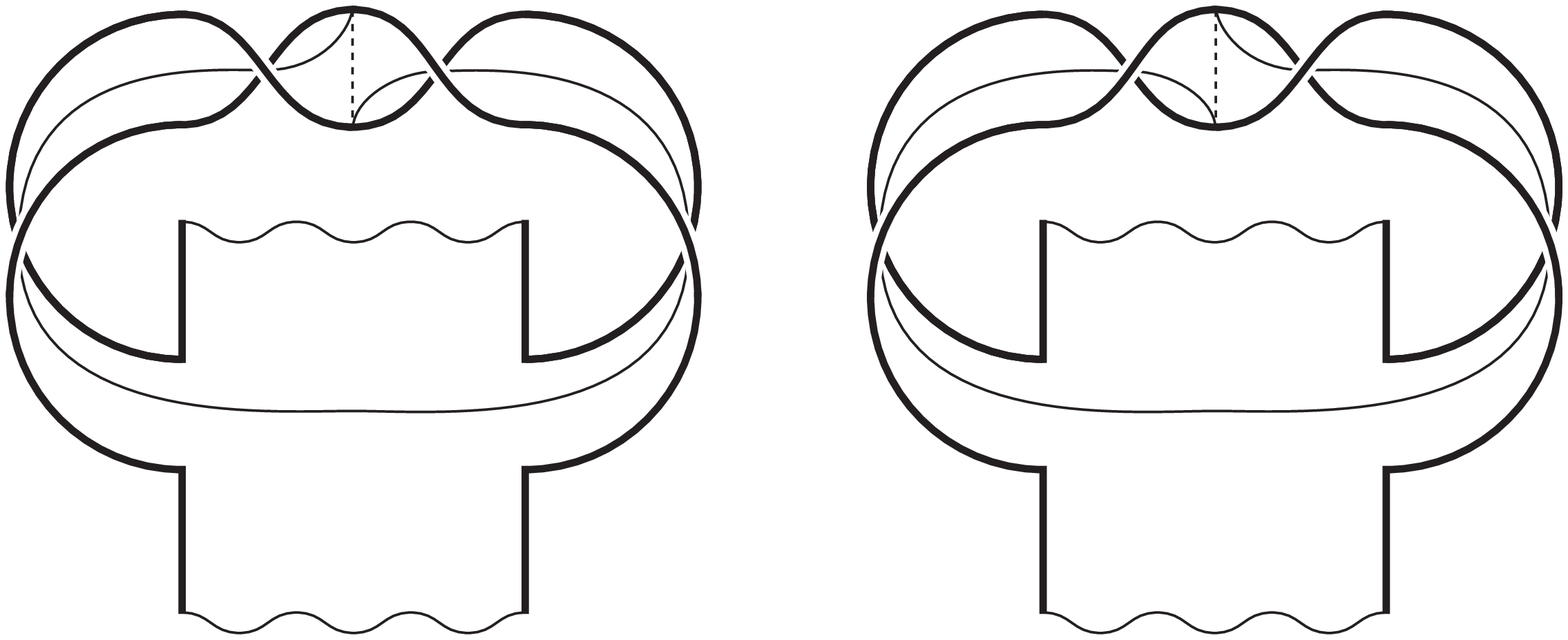}

\begin{picture}(400,0)(0,0)
\put(123,115){$\alpha$}
\put(110,45){$h(\alpha)$}
\put(270,115){$\alpha$}
\put(270,45){$h(\alpha)$}
\put(100,0){$i_{\partial}(\alpha)=1$}
\put(250,0){$i_{\partial}(\alpha)=-1$}
\end{picture}
\caption{Left- and right-handed Hopf bandings of a surface.}
\label{fig:Hopfannuli}
\end{center}
\end{figure}

Next we introduce ``generalized Hopf banding'', which corresponds to a once-unclean non-alternating arc.
We remark that a clean alternating arc $\alpha$ can be moved to be non-alternating 
by adding an unnecessary intersection point with $h(\alpha)$. 
Hence we can say that the band surgery for a once-unclean non-alternating arc is a generalization of Hopf banding.

\begin{definition}
{\rm
Let $\ell$ be an arc in $F'$ such that $\ell$ has a single self-intersection point and $\ell\cap \partial F'=\partial\ell$. 
Let $b$ be a once-overlapped band over $F'$ such that $b([0,1]\times\{\frac{1}{2}\})$ is parallel to $\ell$, see Figure \ref{fig:gHopfband}.
\begin{figure}[h]
\begin{center}
\includegraphics[width=10cm]{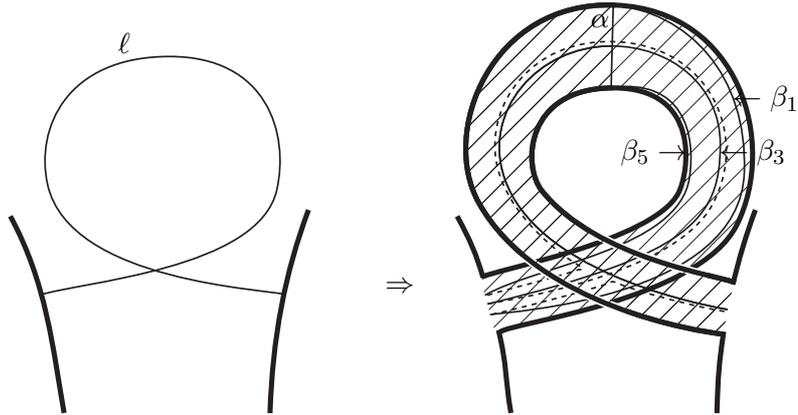}

\begin{picture}(400,0)(0,0)
\put(100,150){$\ell$}
\put(277,160){$\alpha$}
\put(330,130){$\leftarrow\beta_1$}
\put(325,110){$\leftarrow\beta_3$}
\put(288,110){$\beta_5\rightarrow$}
\put(200,60){$\Rightarrow$}
\end{picture}
\caption{Generalized Hopf banding.}
\label{fig:gHopfband}
\end{center}
\end{figure}
If the surface $F$ is obtained by attaching $b$ to $F'$, 
we call $F$ a {\em generalized Hopf banding} of $F'$ along $b$. 
}
\end{definition}

\begin{example} 
By generalized Hopf banding of a Hopf annulus, 
we can obtain two different $3$-component fibered links (see Figure \ref{fig:gHex}).
\begin{figure}[h]
\begin{center}
\includegraphics[scale=.5]{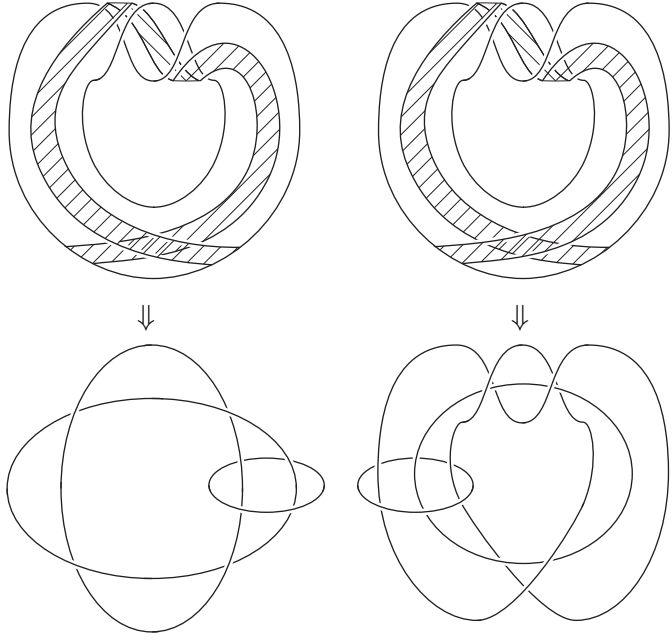}

\begin{picture}(400,0)(0,0)
\put(125,130){$\Downarrow$}
\put(265,130){$\Downarrow$}
\end{picture}
\caption{Generalized Hopf banding of Hopf annulus.}
\label{fig:gHex}
\end{center}
\end{figure}
\end{example}
Note that, for each arc in $F'$ having a self intersection point, 
we have two choices of generalized bandings depending on the overlapped sides. 
Moreover, any Hopf banding is a generalized Hopf banding for $\ell$ 
whose self intersection point is removable by isotopies in $F'$. 
Then we have the following:
\begin{thm}\label{thm:gHopfband}
Suppose $F$ is a fiber surface. 
Then $F$ is not a Hopf banding but a generalized Hopf banding of $F'$ 
if and only if the spanning arc $\alpha$ is once-unclean and non-alternating.
\end{thm}
\begin{proof}
Suppose $F$ is a generalized Hopf banding of $F'$ along a band $b$.
Let $b':[0,1]\times[0,1]\to F$ be a projection of $b:[0,1]\times[0,1]\to M$ into $F'$, 
and put $I_i:=[\frac{i}{5},\frac{i+1}{5}]$ for $i\in\{0,1,2,3,4\}$. 
We may assume that $b'([0,1]\times\{\frac{1}{2}\})=\ell$ and 
$b'|_{I_1\times[0,1]}(\frac{s+1}{5},t)=b'|_{I_3\times[0,1]}(\frac{4-t}{5},s)$ for $(s,t)\in[0,1]\times[0,1]$, 
and so the self intersection of $\ell$ is $b'(\frac{3}{10},\frac{1}{2})=b'(\frac{7}{10},\frac{1}{2})$. 
We also assume that $b(\frac{7}{10},\frac{1}{2})$ is over $b(\frac{3}{10},\frac{1}{2})$.
Let $\beta_1,\beta_2,\beta_3,\beta_4,\beta_5$ be arcs in $F$ 
($\beta_1,\beta_3,\beta_5\subset b([0,1]\times[0,1])$ and $\beta_2,\beta_4,\subset F'$) 
defined by the following (see Figure \ref{fig:gHopfband}): 
\begin{eqnarray*}
\beta_1&:=&\{b(s,\frac{1-2s}{3})\ |\ 0\le s\le\frac{1}{2}\}\\
\beta_2&:=&\{b'(s,\frac{1}{3})\ |\ 0\le s\le\frac{3}{10}\}\cup\{b'(s,\frac{1}{2})\ |\ \frac{11}{15}\le s\le1\}\\
\beta_3&:=&\{b(s,\frac{1}{2})\ |\ 0\le s\le1\}\\
\beta_4&:=&\{b'(s,\frac{2}{3})\ |\ 0\le s\le\frac{3}{10}\}\cup\{b'(s,\frac{1}{2})\ |\ 0\le s\le\frac{2}{3}\}\\
\beta_5&:=&\{b(s,\frac{2+2s}{3})\ |\ 0\le s\le\frac{1}{2}\}
\end{eqnarray*}
Set $\beta:=\beta_1\cup\beta_2\cup\beta_3\cup\beta_4\cup\beta_5$. 
Then $h(\alpha)$ is isotopic to $\beta$ in $F$, since $\beta$ is a proper arc in $F$ with $\partial\beta=\partial\alpha$ 
and $\alpha\cup\beta$ bounds a disk in the complement of $F$. 
The end points of $\beta$ emanate to the same side of $\alpha$ and
$int(\alpha)\cap int(\beta)=b(\frac{1}{2},\frac{1}{2})$. 
Now $b$ is not the band of a Hopf banding, and so $\alpha$ is not clean alternating by Theorem \ref{thm:Hopfband}.
Therefore $\alpha$ is once-unclean and non-alternating.

Suppose now that $\alpha$ is once-unclean and non-alternating.
Set $\beta:=h(\alpha)$.
Then $\beta$ is divided into five arcs $\beta_1,\beta_2,\beta_3,\beta_4,\beta_5$ by cutting along $b(\{0,1\}\times[0,1])$ 
so that $\beta_i$ connects $\beta_{i-1}$ and $\beta_{i+1}$ for $i\in\{1,2,3,4,5\}$, $\beta_0=\beta_6=\alpha$.
We may assume that $\beta_1,\beta_3,\beta_5$ are represented as above.
Set $\ell':=\beta_2\cup \{b(0,s)\ |\ \frac{1}{3}\le s\le\frac{2}{3}\}\cup\beta_4$.  
The arc $\ell'$ attaches to $\partial F'$ at $\{b(0,s)\ |\ \frac{1}{3}\le s\le\frac{2}{3}\}$. 
We have an arc $\ell$ with a single self intersection point by moving $\ell'$ slightly into the interior of $F'$.
Then $F$ is a generalized Hopf banding of $F'$ for $\ell$, and not a Hopf banding by Theorem \ref{thm:Hopfband}.
\end{proof}
\subsection{Generalized Hopf banding for fiber surfaces}
It is well known that a Hopf banding is a fiber surface if and only if 
the original surface is a fiber surface.
In general a resulting surface of a Murasugi sum is a fiber surface if and only if the summands are both fiber surfaces \cite{GabMSNGO,GabDFLS3}.
We have a similar result for generalized Hopf bandings.
\begin{thm}\label{thm:gHopfbandfiber}
Suppose $F$ and $F'$ are surfaces  
such that $F$ is a generalized Hopf banding of $F'$.
Then  $F$ is a fiber surface if and only if $F'$ is a fiber surface.
\end{thm}
\begin{proof}

One direction follows from Theorems \ref{thm:whenfiber}, \ref{thm:Hopfband}, and \ref{thm:gHopfband}. 

We will show that if $F'$ is a fiber surface, then the complimentary sutured manifold $(\overline{M\rmv n(F)},\partial F)$ is trivial, and so $F$ is a fiber surface. 
As in the proof of Theorem \ref{thm:gHopfband},  $\alpha\cup \beta$ bounds a disk in the complement of $F$. 
From the disk, we have the product disk $D$ for $(\overline{M\rmv n(F)},\partial F)$. 
Note that $n(F)$ is obtained from $n(F')$ by attaching a $1$-handle $n(b)$. 
Since $ int(\alpha)$ intersects $int(\beta)$ at a point, and $\beta$ emanates away from $\alpha$ in the same direction at both endpoints of $\alpha$, $\partial D$ intersects a co-core of the $1$-handle at a point, 
and so  $D$ cancels the $1$-handle.
Then  $(\overline{M\rmv n(F)},\partial F)$ is decomposed into $(\overline{M\rmv n(F')},\partial F')$ by $D$, i.e.
$$(\overline{M\rmv n(F)},\partial F)\stackrel{D}{\leadsto}(\overline{M\rmv n(F')},\partial F').$$
Since $F'$ is a fiber surface,  $(\overline{M\rmv n(F')},\partial F')$ is a trivial sutured manifold, 
and so  $(\overline{M\rmv n(F)},\partial F)$ is also trivial. 
Hence $F$ is a fiber surface.
\end{proof}

By Theorems \ref{thm:whenfiber}, \ref{thm:Hopfband}, \ref{thm:gHopfband}, and \ref{thm:gHopfbandfiber}, we have the following: 
\begin{thm}\label{thm:FF'fiber}
{\rm(1)} Suppose $F$ is a fiber surface, 
and $b$ is a band in $F$ such that $b\cap \partial F=b([0,1]\times \{0,1\})$.
Set $F':=\overline{F\rmv b}$.
Then  $F'$ is a fiber surface if and only if $F$ is a generalized Hopf banding of $F'$ along $b$.

{\rm(2)} Suppose $F'$ is a fiber surface 
and $b$ is a band attached to $F'$, 
i.e. $b\cap F'=b(\{0,1\}\times[0,1])\subset\partial F'$. 
Set $F:=F'\cup b$. 
Then $F$ is a fiber surface if and only if $F$ is a generalized Hopf banding of $F'$ along $b$. 
\end{thm}

Theorem \ref{thm:Euler} implies that any coherent band surgery on links can be regarded as 
an operation of cutting a taut Seifert surface along the band.  
Then as a translation of Theorem \ref{thm:FF'fiber}, we have proven Theorem \ref{thm:fiberedlinks}.
\medskip

\noindent
{\bf Theorem \ref{thm:fiberedlinks}.}
{\it
Suppose $L$ and $L'$ are links in $S^3$, and $L'$ is obtained from $L$ by a coherent band surgery
and $\chi(L')=\chi(L)+1$. 

{\rm (1)} Suppose $L$ is a fibered link. Then $L'$ is a fibered link if and only if 
the fiber $F$ for $L$ is a generalized Hopf banding of a Seifert surface $F'$ for $L'$ along $b$. 

{\rm (2)} Suppose $L'$ is a fibered link. Then $L$ is a fibered link if and only if 
a Seifert surface $F$ for $L$ is a generalized Hopf banding of the fiber $F'$ for $L'$ along $b$. }
\medskip

It is well known that any automorphism of a surface 
can be represented by a composition of Dehn twists. 
Let $F$ be a fiber surface with monodromy $h$.  
Honda, Kazez, and Matic \cite{HonKazMatRVDCSB} showed the following:
\begin{lem}\label{lem:alternating}\cite[Lemma 2.5]{HonKazMatRVDCSB}
Suppose $h$ is a composition of right hand Dehn twists along circles in $F$. 
Then $h$ is right-veering, {\it i.e.}
any arc $\alpha$ in $F$ is alternating, or else $h(\alpha)$ is isotopic to $\alpha$ in $F$ ($\alpha$ is non-alternating and clean). 
In other words, $i_{\partial}(\alpha,h(\alpha))=1$ if $h(\alpha)$ is not isotopic to $\alpha$. 
\end{lem}
Remark that $h(\alpha)$ is isotopic to $\alpha$ if and only if there exists a $2$-sphere $S$ such that $S\cap F=\alpha$. 
If we assume additionally that $\partial F$ is prime,  then any essential arc in $F$ is alternating. 
We will discuss the case where $\partial F$ is composite in Subsection \ref{ssec:composite}.

Suppose a fiber surface $F$ with monodromy $h$ is obtained by plumbing of two surfaces 
$F_1$ and $F_2$, where $F_1$ is a Hopf annulus with left hand twist.
Let $C$ be a core circle of $F_1$. 
We denote by $t_C$ the right hand Dehn twist along $C$.
Then $(t_C^{-1}\circ h)|_{F_2}$ is isotopic to the monodromy for $F_2$. 
Hence if $F$ is obtained from a disk in $S^3$ by successively plumbing Hopf annuli with right hand twist, 
then $h$ is a composition of right hand Dehn twists.
By Theorem \ref{thm:fiberedlinks} (1) and Lemma \ref{lem:alternating}, 
we have the following:
\begin{cor}\label{cor:Hopfplumbing}
Let $L$ be an oriented link in $S^3$ with fiber $F$ such that $F$ is obtained from a disk
by successively plumbing Hopf annuli with right hand twists 
(or by successively plumbing Hopf annuli with left hand twists). 
Suppose $L'$ is a link obtained from $L$ by a coherent band surgery and $\chi(L')=\chi(L)+1$. 
Then $L'$ is a fibered link if and only if 
$F$ is a Hopf banding of a Seifert surface $F'$ for $L'$ along $b$. 
\end{cor}

\begin{rem}
\label{rem:BaaderDehornoy}
Baader and Dehornoy have just announced a similar result in \cite{BaaDehTP}.
\end{rem}

\subsection{Band surgeries on $(2,p)$-torus link}\label{ssec:(2,p)}
Let $D_1$ and $D_2$ be disjoint disks in a plane.  
Let $b_1,\ldots,b_p$ be pairwise disjoint bands, each with a left hand half twist, connecting the two disks. 
Set $F:=D_1\cup D_2\cup b_1\cup\cdots\cup b_p$. 
Then $F$ is a fiber surface for the $(2,p)$-torus link $T(2,p)$ (with parallel orientation if $p$ is even) (see Figure \ref{fig:torussurface}). 
\begin{figure}[h]
\begin{center}
\includegraphics[width=10cm]{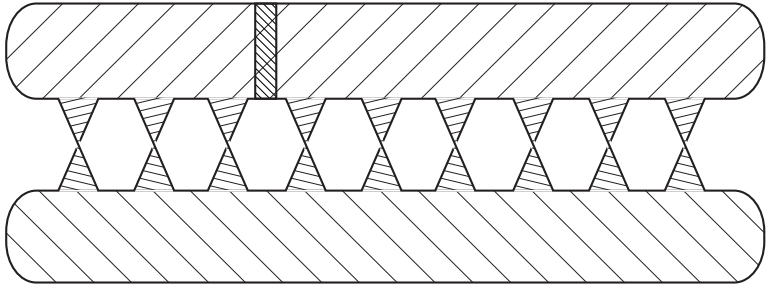}
\caption{Fiber surface for $(2,9)$-torus knot}
\label{fig:torussurface}
\end{center}
\end{figure}
Let $b$ be a band in $F$, and set $F':=\overline{F\rmv b}$, $L':=\partial F'$. 
Since $F$ is obtained from $D_1\cup D_2\cup b_1\cup\cdots\cup b_{p-1}$ by plumbing a Hopf annulus with a left hand twist along $b_{p-1}$ (or its spanning arc), 
$F$ is obtained from a disk $D_1\cup D_2\cup b_1$ by successively plumbing $(p-1)$ Hopf annuli with left hand twists. 
Then, by Corollary \ref{cor:Hopfplumbing}, $L'$ is fibered if and only if $F$ is a Hopf banding of $F'$ along $b$.

\begin{cor}\label{cor:2torus}
Suppose $L'$ is obtained from $L=T(2,p)$ by a coherent band surgery along $b$, where $p\ge 2$, and $\chi(L')>\chi(L)$.  
Then $L'$ is fibered if and only if the band $b$ can be moved into $D_1$ (and also $D_2$) 
so that 
$F \rmv b$ is connected.
In particular, if we assume $L'$ is a prime (resp. composite) fibered link, 
then $L'$ is $T(2,p-1)$ (resp. a connected sum $T(2,p_1)\# T(2,p_2)$ of $T(2,p_1)$ and $T(2,p_2)$,  
where $p_1$ and $p_2$ are positive integers with $p_1,p_2>1$ and $p_1+p_2=p$). 
Moreover,
the band is unique up to isotopy fixing $L$ as a set if $L'=T(2,p-1)$ or $T(2,p_1)\# T(2,p_2)$ and either $p_1$ or $p_2$ is odd, 
and there are two bands up to isotopy fixing $L$ as a set if both $p_1$ and $p_2$ are even 
($L'$ is a $3$-component link), 
but they are the same up to homeomorphism.
\end{cor}
\begin{rem}
By Murasugi \cite{MurOCNILT}, $|\sigma(L)-\sigma(L')|\le 1$ for two links $L,L'$ which are related by a coherent band surgery, 
where $\sigma$ means the signature. 
Since $\chi(T(2,p))=2-p$ and $\sigma(T(2,p))=1-p$, 
$\chi(L)+1=\sigma(L)\ge\sigma(L')-1\ge\chi(L')$ 
if $L=T(2,p)$.
Then the assumption $\chi(L')>\chi(L)$ in Corollary \ref{cor:2torus} becomes $\chi(L')=\chi(L)+1$.
Remark that we can regard $T(2,p-1)$ as $T(2,p_1)\#T(2,p_2)$ for $p_1=p-1$ and $p_2=1$ since $T(2,1)$ is trivial.
\end{rem}

\begin{proof}
Suppose that $b$ is contained in $F$, disjoint from $b_1,\ldots,b_p$, and does not separate $F$.
We will prove that $F'$ is fibered, and that the band is unique up to the operations mentioned.
Say $b$ is contained in $D_1$, 
and $b$ splits $D_1$ into two disks with $p_i$ bands of $b_1,\ldots,b_p$ ($i=1,2$), 
where $p_1$ and $p_2$ are positive integers with $p_1+p_2=p$. 
Then $L'$ is a connected sum $T(2,p_1)\# T(2,p_2)$ of $T(2,p_1)$ and $T(2,p_2)$ which is a fibered link. 
Moreover two such bands in $F$ are related by the monodromy and sliding along $\partial F$ 
if the two bands are attached to the same component of $L$. 
This implies that the band is unique up to isotopies fixing $L$ as a set if either $p_1$ or $p_2$ is odd. 
If the two bands are attached to different components of $L$, they are related by the monodromy, sliding along $\partial F$, 
and an involution. 
Here we can take a rotation about the horizontal axis in Figure \ref{fig:torussurface} as the involution 
so that $D_1$ is mapped to $D_2$,  $D_2$ is mapped to $D_1$, and $b_i$ is mapped to itself. 
This implies that the two bands are the same up to homeomorphism.

Conversely, let $\alpha$ be a clean and alternating arc in $F$. 
We will show that $\alpha$ can be moved  into $D_1$ or $D_2$ so that $\alpha$ is disjoint from $b_1,\ldots,b_p$. 
This will show that any band producing a fibered link $L'$ can be moved into $D_1$ or $D_2$
by Corollary \ref{cor:Hopfplumbing} and Theorem \ref{thm:Hopfband}.
We arrange the bands $b_1,\ldots, b_p$ along an orientation of $\partial D_1$ (or $\partial D_2$).
For each $i\in\{1,2\}$ and $j\in \{1,\ldots,p\}$, let $\delta_{ij}=\partial D_i\cap\partial b_j$ be an arc with the orientation induced by that of $D_j$.  Note that $\delta _{1j}$ and $\delta _{2j}$ are isotopic to each other in $F$ but having opposite orientations.  
It is well known that the monodromy $h$ of $F$ is represented by $t_1\circ t_2\circ\cdots\circ t_{p-1}$, where $t_i$ is a Dehn twist along a loop in $F$ passing only once through each of $b_i, D_1, b_{i+1}$, and $D_2$. 
Then we can see that $h(\delta_{1j})$ is isotopic to $\delta_{2(j+1)}$ (similarly, $h(\delta_{2j})$ is isotopic to $\delta_{1(j+1)}$), including the orientation, by sliding to the left hand side along $\partial F$. 
Let $\widehat{h}$ be an automorphism of $F$ such that 
$\widehat{h}(D_1)=D_2$, $\widehat{h}(D_2)=D_1$, 
and $\widehat{h}(b_i)=b_{i+1} (\mbox{mod } p )$.
Then $\widehat{h}$ is obtained from $h$ by sliding to the left hand side along $\partial F$.
Since $\alpha$ and $h(\alpha)$ intersect only at their endpoints with positive signs, 
$\alpha$ is disjoint from $\widehat{h}(\alpha)$. 
We may assume that $\alpha$ minimizes the number of  intersections with $int (b_1\cup\cdots\cup b_p)$,  
and $\partial\alpha$ consists of two points of $(\partial \delta_{11}\cup\cdots\cup\partial \delta_{1p})\cup(\partial \delta_{21}\cup\cdots\cup\partial \delta_{2p})$.  
For a contradiction, suppose $\alpha$ intersects $int (b_1\cup\cdots\cup b_p)$.
Then $\alpha$ is divided into arcs by cutting $F$ along $b_1\cup\cdots\cup b_p$. 
Let $\alpha_1,\alpha_2$ be successive such arcs in $D_1,D_2$ respectively, and define the following (see Figure \ref{fig:arcsinF}):\\
(1) $\partial \alpha_1=\{x,y\}$, where $x$ is a point in $\delta_{1i}$ and $y$ is a point in $\delta_{1j}$.\\
(2) $\partial \alpha_2=\{z,w\}$, where $z$ is a point in $\delta_{2j}$ and $w$ is a point in $\delta_{2k}$.\\
(3) A component of $\alpha\cap b_j$ connects $y$ and $z$ in $b_j$.\\
\begin{figure}[h]
\begin{center}
\includegraphics[width=4in]{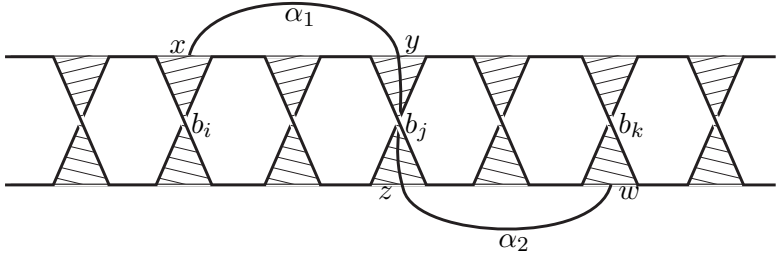}

\begin{picture}(400,0)(0,0)
\put(125,50){$b_{i}$}
\put(205,50){$b_{j}$}
\put(285,50){$b_k$}
\put(117,80){$x$}
\put(205,82){$y$}
\put(195,25){$z$}
\put(285,25){$w$}
\put(160,92){$\alpha_1$}
\put(240,7){$\alpha_2$}
\end{picture}
\caption{If the arc $\alpha$ intersects the bands $b_1 \cup \cdots \cup b_p$, then it is divided into sub-arcs by cutting along the bands.}
\label{fig:arcsinF}
\end{center}
\end{figure}
Set $\beta_1:=\widehat{h}(\alpha_1),\beta_2:=\widehat{h}(\alpha_2)$, 
$x':=\widehat{h}(x),y':=\widehat{h}(y),z':=\widehat{h}(z),w':=\widehat{h}(w)$. 
Then $x',y',z',w'$ are points in $\delta_{2(i+1)},\delta_{2(j+1)},\delta_{1(j+1)},\delta_{1(k+1)}$ respectively.

First we show that $j-i\equiv\pm1$ (mod $p$) or $k-j\equiv\pm1$ (mod $p$). 
Suppose $k-j\not\equiv\pm1$ (mod $p$). 
Let $D_1'$  (resp., $D_2'$) be a disk cut off from $D_1$ by $\beta_2$ (resp., $D_2$ by $\alpha_2$),  
where $\partial D_1'\cap(b_1\cup\cdots\cup b_p)=(\partial D_1'\cap \delta_{1(k+1)})\cup(\delta_{1(k+2)}\cup\cdots\cup\delta_{1j})\cup(\partial D_1'\cap \delta_{1(j+1)})$
(resp., $\partial D_2'\cap(b_1\cup\cdots\cup b_p)=(\partial D_2'\cap \delta_{2j})\cup(\delta_{2(j+1)}\cup\cdots\cup\delta_{2(k-1)})\cup(\partial D_2'\cap \delta_{1k})$).
Since $\alpha_1$ is disjoint from $\beta_2$ in $D_1$, 
two points $x$ and $y$ both lie in $\partial D_1'$,  
and so $i\equiv k+1,k+2,\ldots, j-1$, or $j+1$ (mod $p$). 
On the other hand, since $\alpha_2$ is disjoint from $\beta_1$ in $D_2$, 
two points $x'$ and $y'$ both lie in $\partial D_2'$, 
and so $i+1\equiv j,j+2,j+3,\ldots, k-1$, or $k$ (mod $p$). 
Then $j-i\equiv\pm1$ (mod $p$).

Next we show that if $\alpha_1$ is outermost in $D_1$ and $j-i\equiv 1$ (mod $p$),
then $\alpha_2$ is outermost in $D_2$ and $k-j\equiv 1$ (mod $p$).  
Similarly, if $\alpha_2$ is outermost in $D_2$ and $k-j\equiv -1$ (mod $p$),
then $\alpha_1$ is outermost in $D_1$ and $j-i\equiv -1$ (mod $p$). 
Suppose that $\alpha_1$ is outermost in $D_1$ and $j=i+1$ ($j=1$ if $i=p$) . 
Then $\beta_1$ connects a point $x'$ in $\delta_{2(i+1)}$ 
and a point $y'$ in $\delta_{2(i+2)}$. 
Recall that $z$ is a point in $\delta_{2(i+1)}$.
Since $\alpha_1$ is outermost in $D_1$, $z$ lies in the side of $\beta_1$ containing no $\delta_{2l}$'s, 
and so $\alpha_2$ is parallel to $\beta_1$ and is outermost in $D_2$. 

Finally we show that this results in a contradiction. 
Suppose $\alpha$ has a sub-arc $\alpha'$ which is outermost in $D_1$ or $D_2$ 
and connecting two adjacent bands. 
By continuing the same argument above, we may assume that 
the outermost sub-arc $\alpha_2$ of $\alpha$ is outermost 
in $D_1$ or $D_2$, say $D_2$, and $k-j\equiv 1$ (mod $p$).
Since $\widehat{h}(\alpha)$ passes through $b_{j+1}$, there exists a sub-arc $\ell=\ell_1\cup\ell_2$ of $L$  
such that $\ell_1$ and $\ell_2$ are components of $L\cap\partial b_j$ and $L\cap\partial D_2$ respectively, 
and $\ell_2\cap b_{j+1}=\partial\alpha_2\cap\partial\delta_{2(j+1)}$ is an endpoint of $\alpha$.
Then an arc component of $\alpha\cap int(b_j)$ is removable by sliding $\alpha$ along $\ell$. 
In the case where $\widehat{h}(\alpha)$ has a sub-arc which is outermost in $D_1$ or $D_2$ 
and connecting two adjacent bands, by the same argument, $\widehat{h}(\alpha)$ (and so does $\alpha$)
has a removable intersection with $int(b_1\cup\cdots\cup b_p)$. 
This contradicts the assumption that $\alpha$ minimizes the number of intersections with $int(b_1\cup\ldots\cup b_p)$. 
\end{proof}

\subsection{Band surgeries on composite fibered links}\label{ssec:composite}

We say that a fiber surface is {\em prime} (resp., {\em composite}) if  the boundary is a prime link 
(resp., a composite link).
Suppose $F$ is a composite fiber suface.
There exists a $2$-sphere $S$ intersecting $F$ in an arc, 
such that neither surface cut off from $F$ by the arc is a disk. 
The resulting surfaces are both fiber surfaces
for the summand links. 
In general, there exist pairwise disjoint $2$-spheres $S_1,\ldots,S_m$ 
such that $\delta_i:=S_i\cap F$ is an arc for each $i\in\{1,\ldots,m\}$, 
and each component of the surface obtained from $F$ by cutting along $\delta_1,\ldots,\delta_n$ 
is a prime fiber surface. 
We call a set $\{\delta_1,\ldots,\delta_m\}$ of such arcs a {\em full prime decomposing system} for $F$. 
We remark that  if $m=1$, a full prime decomposing system (an arc in this case) is unique up to isotopy in $F$. 
On the other hand, there may exist several decomposing systems 
if $m\ge2$, 
but the sets of surfaces obtained from $F$ by cutting along decomposing systems are always the same. 

Suppose a fiber surface $F$ is divided into 
prime fiber surfaces $F_1,\ldots,F_{m+1}$, 
and a properly embedded arc $\alpha$ in $F$ is divided minimizingly into sub-arcs $\alpha_1,\ldots, \alpha_n$ successively by 
a full prime decomposing system $\{\delta_1,\ldots,\delta_m\}$, 
where $\alpha_i$ is a properly embedded arc in $F_{j_i}$ for each $i\in\{1,\ldots,n\}$ and 
$\{p_i\}=\alpha_i\cap\alpha_{i+1}\subset\partial\alpha_i,\partial\alpha_{i+1}$ for each $i\in\{1,\ldots,n-1\}$. 
Let $s_i,t_{i+1}=\pm1$ be the signs at $p_i$ for a pair $(\alpha_i,h_{j_i}(\alpha_i))$  in $F_{j_i}$ 
and for a pair $(\alpha_{i+1}, h_{j_{i+1}}(\alpha_{i+1}))$ in $F_{j_{i+1}}$ respectively. 
Remark that $i_{\partial}(\alpha_i,h_{j_i}(\alpha_i))=\frac{t_i+s_i}{2}$ for each $i\in\{2,\ldots,n-1\}$, see \cite{GooOOBSA}. 
Then we have the following.
\begin{lem}\label{lem:dividedarcs}
 $$\rho(\alpha)=\sum_{i=1}^n\rho(\alpha_i)+\frac{1}{2}\sum_{i=1}^{n-1}|s_i+t_{i+1}|$$
Here if $h_{j_i}(\alpha_i)$ is isotopic to $\alpha_i$ in $F_{j_i}$, 
$(t_i,s_i)=(1,-1)$ or $(-1,1)$ 
which minimizes $\sum_{i=1}^{n-1}|s_i+t_{i+1}|$.
\end{lem}
\begin{proof}
First we will show that $\rho(\alpha)\le\sum_{i=1}^n\rho(\alpha_i)+\frac{1}{2}\sum_{i=1}^{n-1}|s_i+t_{i+1}|$.
We can take the monodromies $h,$ $h_1,\ldots,h_{m+1}$ of $F$, $F_1,\ldots,F_{m+1}$ respectively,  
and $\alpha_1,\ldots,\alpha_n$ 
so that $h|_{F_j}=h_j$ 
and $|int(\alpha_i)\cap int(h(\alpha_i))|=\rho(\alpha_i)$.
By moving $h$ slightly at $p_i$ if $s_i+t_{i+1}=0$, then, we have 
$\rho(\alpha)\le|int(\alpha)\cap int(h(\alpha))|=\sum_{i=1}^n\rho(\alpha_i)+\frac{1}{2}\sum_{i=1}^{n-1}|s_i+t_{i+1}|$.

Next we will show that $\rho(\alpha)\ge\sum_{i=1}^n\rho(\alpha_i)+\frac{1}{2}\sum_{i=1}^{n-1}|s_i+t_{i+1}|$. 
Put $\beta:=h(\alpha)$  
so that $|int(\alpha)\cap int(\beta)|=\rho(\alpha)$. 
Then there exists a disk $D$, possibly with self intersection in the boundary, 
such that $D\cap F=\partial D=\alpha\cup\beta$. 
We analyze the intersection of $D$ and the pairwise disjoint spheres $S_1,\ldots,S_m$, where $S_i\cap F=\delta_i$ for each $i\in\{1,\ldots,m\}$. 
By a cut and paste argument, we may assume that the intersection $D\cap (S_1\cup\cdots\cup S_m)$ consists of arcs. 
Let $D'$ be an outermost disk cut off from $D$ by $D\cap S_i$ for some $i\in\{1,\dots,m\}$. 
Suppose $\partial D'\cap\partial D\subset \alpha$ or $\partial D'\cap\partial D\subset \beta$. 
Since $F$ is incompressible, $\partial D'\cap\partial D$ is isotopic in $F$ to 
a sub-arc of $\delta_i$ joining the end points of the arc of $D\cap S_i$. 
Hence such an arc of intersection $D\cap S_i$ is removable keeping $|int(\alpha)\cap int(\beta)|$ constant. 
After removing such intersections, $D$ is divided into disks $D_1,\ldots,D_n$ by $D\cap (S_1\cup\cdots\cup S_m)$,  
where $\partial D_i$ consists of $\alpha_i$, a sub-arc $\beta_i$ of $\beta$,
and two parallel arcs (resp., a single arc) of $D\cap (S_1\cup\cdots\cup S_m)$  for each $i\in\{2,\ldots,n-1\}$ (resp., $i\in\{1,n\}$) (see Figure \ref{fig:DcapS}). 
This implies that $h_{j_i}(\alpha_i)$ is isotopic to $\beta_i$. 
By sliding $\beta_i$ along $\partial F_{j_i}$ in $F_{j_i}$ so that the end points of $\beta_i$ coincide with those of $\alpha_i$, 
we have 
$\rho(\alpha)=|int(\alpha)\cap int(\beta)|=\sum_{i=1}^n|int(\alpha_i)\cap int(\beta_i)|+\frac{1}{2}\sum_{i=1}^{n-1}|s_i+t_{i+1}|\ge\sum_{i=1}^n\rho(\alpha_i)+\frac{1}{2}\sum_{i=1}^{n-1}|s_i+t_{i+1}|$. 
 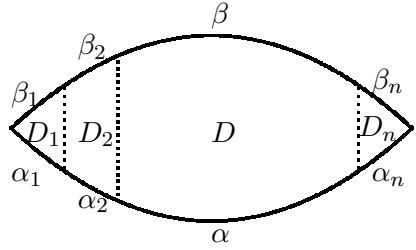
\begin{figure}[h]
\begin{center}
\begin{picture}(150,80)(0,0)
\put(5,25){$D_1$}
\put(25,25){$D_2$}
\put(130,27){$D_n$}
\put(0,10){$\alpha_1$}
\put(0,40){$\beta_1$}
\put(25,0){$\alpha_2$}
\put(25,58){$\beta_2$}
\put(135,10){$\alpha_n$}
\put(135,45){$\beta_n$}
\put(75,-13){$\alpha$}
\put(75,70){$\beta$}
\put(75,25){$D$}
\linethickness{0.3mm}
\qbezier[10](20,15)(20,30)(20,45)
\qbezier[10](130,15)(130,30)(130,45)
\qbezier[20](40,5)(40,30)(40,55)
\linethickness{0.3mm}
\qbezier(0,30)(75,-40)(150,30)
\qbezier(0,30)(75,100)(150,30)
\linethickness{0.5mm}
\end{picture}
\end{center}
\caption{The disk $D$ is divided into sub-disks by the arcs of intersection with the prime decomposing spheres $S_1 \cup \cdots \cup S_m$.}
\label{fig:DcapS}
\end{figure}
\end{proof}

By Lemma \ref{lem:alternating} and Lemma \ref{lem:dividedarcs}, we have the following.
\begin{thm}\label{thm:connectedsum+-}
Suppose that a fiber surface $F$ is divided into fiber surfaces $F_1,\ldots,F_{m+1}$, 
and a properly embedded arc $\alpha$ in $F$ is divided minimizingly into sub-arcs $\alpha_1,\ldots, \alpha_n$ successively by 
a full prime decomposing system $\{\delta_1,\ldots,\delta_m\}$, 
where  
the monodromy $h_j$ of $F_j$ is a composition of right hand Dehn twists
or left hand Dehn twists according to whether $\varepsilon_j=+$ or $-$,
and $\alpha_i$ is a properly embedded arc in $F_{j_i}$. 
Then the following holds:

\begin{enumerate}[\rm(1)]
\item The arc $\alpha$ is clean and alternating if and only if 
the set $\{1,\ldots,n\}$ is partitioned into two sets $A$ and $B$, with $A$ consisting of an odd number of elements, so that
$\alpha_i$ is clean and alternating in $F_{j_i}$ 
and $\varepsilon_{j_i}$ appears as $+$ and $-$ alternately in ascending order for $i\in A$, 
and $\alpha_i$ is parallel to the boundary $\partial F_{j_i}$ in $F_{j_i}$ for any $i\in B$.

\item The arc $ \alpha$ is  once-unclean and non-alternating if and only if either:

\begin{enumerate}[\rm(2-1)]
\item the set $\{1,\ldots,n\}$ is partitioned into two sets $A$ and $B$, with $A$ consisting of an even number of elements, so that 
$\alpha_i$ is clean and alternating in $F_{j_i}$ 
except for one once-unclean alternating arc
and $\varepsilon_{j_i}$ appears as $+$ and $-$ alternately in ascending order for $i\in A$, 
and $\alpha_i$ is parallel to the boundary $\partial F_{j_i}$ in $F_{j_i}$ for any $i\in B$, or

\item the set $\{1,\ldots,n\}$ is partitioned into two sets $A$ and $B$, with $A$ consisting of an odd number of elements, so that 
$\alpha_i$ is clean and alternating in $F_{j_i}$ 
and $\varepsilon_{j_i}$ appears as $+$ and $-$ alternately except one successive pair
in ascending order for $i\in A$, 
and $\alpha_i$ is parallel to the boundary $\partial F_{j_i}$ in $F_{j_i}$ for any $i\in B$.
\end{enumerate}
\end{enumerate}
\end{thm}
\begin{proof}
By Lemma \ref{lem:alternating}, 
$\alpha_i$ is non-alternating ($t_i+s_i=0$) if and only if  $h_{j_i}(\alpha_{j_i})$ is isotopic to $\alpha_i$ in $F_{j_i}$, 
and if $\alpha_i$ is alternating ($t_i+s_i\neq 0$) then $t_i=s_i=\varepsilon_{j_i}$. 
Since $F_{j_i}$ is prime,  
$h_{j_i}(\alpha_{i})$ is isotopic to $\alpha_i$ in $F_{j_i}$ if and only if 
$\alpha_i$ is parallel to the boundary $\partial F_{j_i}$ in $F_{j_i}$.
Partition the set $\{1,\ldots,n\}$ into $A$ and $B$ so that $\alpha_i$ is alternating if $i\in A$, 
and parallel to the boundary $\partial F_{j_i}$ in $F_{j_i}$ if $i\in B$.
Suppose that $i\in A$ and $i+1\in B$ (resp., $i-1\in B$), 
then $\rho(\alpha_{i})=0$, and $(t_i,s_i)$ can be taken as $(\varepsilon_{i+1},-\varepsilon_{i+1})$ 
(resp., $(-\varepsilon_{i-1},\varepsilon_{i-1})$) 
so that $s_i$ and $t_{i+1}$ (resp., $s_{i-1}$ and $t_i$) cancel. 
Hence we can ignore the elements of $B$ when we calculate the the number $\sum_{i=1}^n\rho(\alpha_i)+\frac{1}{2}\sum_{i=1}^{n-1}|s_i+t_{i+1}|$.
We remark that  $i_{\partial}(\alpha,h(\alpha))=\varepsilon_{j_i}+\varepsilon_{j_{i'}}$, where $i$ and $i'$ are the first and the last elements of $A$.

(1) By definition, the arc $\alpha$ is clean and alternating if and only if 
$\rho(\alpha)=0$ and $i_{\partial}(\alpha,h(\alpha))=\pm1$.
By Lemma \ref{lem:dividedarcs}, 
$\rho(\alpha)=0$ if and only if $\alpha_i$ is clean for each $i\in A$, 
and $\varepsilon_{j_i}+\varepsilon_{j_{i'}}=0$ for each pair of successive integers $i$ and $i'$ in $A$.
Then (1) of Theorem \ref{thm:connectedsum+-} holds. 

(2) By the definition, the arc $\alpha$ is once-unclean and non-alternating if and only if 
$\rho(\alpha)=1$ and $i_{\partial}(\alpha,h(\alpha))=0$.
By Lemma \ref{lem:dividedarcs}, 
$\rho(\alpha)=1$ if and only if either: 
(2-1) $\alpha_i$ is clean for $i\in A$ except one once-unclean,
and $\varepsilon_{j_i}+\varepsilon_{j_{i'}}=0$ for a pair of successive integers $i$ and $i'$ in $A$, 
or
(2-2) $\alpha_i$ is clean for $i\in A$. 
and $\varepsilon_{j_i}+\varepsilon_{j_i'}=0$ for a pair of successive integers $i$ and $i'$ in $A$ except for one pair.
Then (2) of Theorem \ref{thm:connectedsum+-} holds. 
\end{proof}

The following corollary is derived from Theorem \ref{thm:connectedsum+-} by considering the case when 
$\varepsilon_1=\cdots=\varepsilon_{m+1}=+$ 
or $\varepsilon_1=\cdots=\varepsilon_{m+1}=-$.

\begin{cor}\label{cor:disjointfromsystem}
Suppose a fiber surface $F$ is composite,
has monodromy which is a composition of right hand Dehn twists
or left hand Dehn twists, 
and that an arc $\alpha$ in $F$ is clean and alternating. 
Then there exists a full prime decomposing system $\{\delta_1,\ldots,\delta_m\}$
such that $\alpha$ is disjoint from $\delta_1\cup\cdots\cup\delta_m$.
\end{cor}

\begin{proof}
Let $\{\delta_1,\ldots,\delta_m\}$ be a full prime decomposing system for $F$, 
so that $F$ is divided into fiber surfaces $F_1,\ldots,F_{m+1}$ by 
$\{\delta_1,\ldots,\delta_m\}$. 
Suppose that a clean alternating arc $\alpha$ in $F$ is divided minimizingly 
into sub-arcs $\alpha_1,\ldots, \alpha_n$ ($n\ge2$) successively by $\{\delta_1,\ldots,\delta_m\}$, 
{\it i.e.} $|\alpha\cap (\delta_1\cup\cdots\cup\delta_m)|=n-1$. 
Then the monodromy $h_j$ of $F_j$ is a composition of right hand Dehn twists
or left hand Dehn twists according to whether that of $F$ is a composition of right hand Dehn twists
or left hand Dehn twists. 
By Theorem \ref{thm:connectedsum+-} (1),  there exists $k\in \{1,\ldots,n\}$ such that 
$\alpha_k$ is clean and alternating, 
and any other arc $\alpha_i$ ($i\in\{1,\ldots,n\}-\{k\}$) is parallel to $\partial F_{j_i}$ 
in  $F_{j_i}$
({\it i.e.} the set $A$ in Theorem \ref{thm:connectedsum+-} (1) must be a singleton set $\{k\}$ in this case). 
Without loss of generality, 
we may assume that $k\neq n$,  
$\alpha_{n-1}$ and $\alpha_{n}$ are arcs in $F_{m}$ and $F_{m+1}$ respectively, 
and $F_{m}\cap F_{m+1}=\delta_m$. 
Since $\alpha_n$ is parallel to $\partial F_{m+1}$ in $F_{m+1}$, 
$\alpha_n$ divides $F_{m+1}$ into a disk $D$ and a surface $F'_{m+1}$ which is homeomorphic to $F_{m+1}$, 
and divides $\delta_m$ into $a$ and $b$, where $a\subset \partial F'_{m+1}$ and $b\subset \partial D$. 
Let $\delta'_{m}$ be an arc obtained from $\alpha_n\cup a$ by pushing slightly into the interor of $F'_{m+1}$. 
Then the set $\{\delta_1,\ldots,\delta_{m-1},\delta'_{m}\}$ is a new full prime decomposing system 
which divides $F$ into $F_1,\ldots,F_{m-1},F'_m,F'_{m+1}$, 
where $F'_m=F_m\cup D$, and 
$|\alpha\cap (\delta_1\cup\cdots\cup\delta_{m-1}\cup\delta'_m)|< n-1$. 
By continuing such operations, we obtain a full prime decomposing system for $F$ 
which is disjoint from $\alpha$.
\end{proof}

Hence, in Corollary \ref{cor:Hopfplumbing}, if we assume that $L$ is composite, 
we can take decomposing spheres for $L$ so that the band of a Hopf banding is disjoint from the decomposing spheres. 
Then we have the following from Corollary \ref{cor:2torus}.

\begin{cor}\label{cor:consumtorus}
Suppose $L'$ is obtained from $L=T(2,p)\#T(2,q)$ by a coherent band surgery along $b$ and $\chi(L')>\chi(L)$, 
where $p,q>1$. 
If $L'$ is fibered,  
then $L'$ is a connected sum $T(2,p_1)\# T(2,p_2)\#T(2,q)$ or $T(2,p)\# T(2,q_1)\#T(2,q_2)$,  
where $p_1,p_2$, $q_1,q_2$  are positive integers with $p_1+p_2=p$, $q_1+q_2=q$. 
Moreover, for each $L'=T(2,p_1)\# T(2,p_2)\#T(2,q)$ or $T(2,p)\# T(2,q_1)\#T(2,q_2)$, 
the band  
is unique up to homeomorphisms.
\end{cor}

%% file: loop.tex
Let $L$ be a fibered link in a manifold $M$ with fiber $F$, and let $\alpha$ be an arc in $F$. 
There exists a disk $D$ in $M$ such that $D\cap F=\alpha$ and  $\partial D$ is disjoint from $F$. We call  $c=\partial D$ an \emph{$\alpha$-loop}, or generally an \emph{arc-loop}.
In this section we will characterize Dehn surgeries along arc-loops preserving $F$ as a fiber surface,  using results of Ni \cite{NiDSKPM} about surgeries on knots in trivial sutured manifolds. In Section \ref{section:crossingchanges} and Section \ref{section:alternativeproof},  we will use this characterization for that of generalized crossing changes and an alternative proof of Theorem \ref{thm:whenfiber}, respectively.

\begin{thm}[\cite{NiDSKPM}]
\label{thm:Ni}
Suppose $F$ is a compact surface and that $c \subset F \times I$ is a simple closed curve. Suppose that $\gamma$ is a non-trivial slope on $c$, and that $N(\gamma)$ is the manifold obtained from $F \times I$ via the $\gamma$-surgery on $c$. If the pair $(N(\gamma), (\bd F) \times I)$ is homeomorphic to the pair $(F \times I, (\bd F) \times I)$, then one can isotope $c$ such that its image on $F$ under the natural projection $p: F \times I \to F$ has either no crossing or exactly 1 crossing.

The slope can be determined as follows: Let $\lambda_b$ be the frame specified by the surface $F$. When the projection has no crossing, $\gamma = \frac{1}{n}$ for some integer $n$ with respect to $\lambda_b$; when the minimal projection has exactly 1 crossing, $\gamma = \lambda_b$.
\end{thm}

\begin{rem} \label{rem:Ni} 
Conversely, the surgeries in the statement of Theorem \ref{thm:Ni} do not change the homeomorphism type of the pair $(F\times I, (\bd F) \times I)$.
\end{rem}

Our first objective will be to relate such a loop $c$ to an arc-loop.

\begin{definition}
In $F \times [0,1]$, a loop $c$ is said to be in \emph{1-bridge position} (w.r.t. $x_1,x_2$) if $c$ is partitioned into arcs, $\tau, \beta, \nu_1, \nu_2$, where $\tau$ is embedded in $F \times \set{\frac{1}{3}}$, $\beta$ is embedded in $F \times \set{\frac{2}{3}}$, and $\nu_i = \set{x_i} \times \left [ \frac{1}{3}, \frac{2}{3} \right ] \, \,  (i = 1,2)$.

We extend this definition to a loop $c$ in the complement of a fibered link $L$ in a manifold $M$ with fiber surface $F$ if $c$ is in 1-bridge position in the product structure of the complementary sutured manifold of $F$.
\end{definition}

\begin{definition} We will say that two loops $c$ and $c'$ in 1-bridge positions are \emph{1-bridge isotopic} if there is an isotopy from $c$ to $c'$ so that the curves are in 1-bridge positions throughout the transformation (where the points $x_1, x_2$ may change throughout).
\end{definition}

Recall that $p: F \times I \to F$ is the natural projection map defined by $p(x, t) = x$. The \emph{1-bridge crossing number} of a loop $c$ in 1-bridge position,  $bc_1(c)$, is the minimum number of crossings of $p(c)$ over all 1-bridge positions that are 1-bridge isotopic to $c$.
The \emph{minimum 1-bridge crossing number} of a loop $c$ having 1-bridge positions,  $mbc_1(c)$, is the minimum of the 1-bridge crossing number over all 1-bridge positions that are isotopic to $c$.
We will show that $mbc_1(c)=bc_1(c)$ for any loop $c$ in 1-bridge position.
\begin{lem}
\label{lem:1BridgeFxI}
Let $c$ and $c'$ be loops in $F \times I$ in 1-bridge positions.
If $c$ and $c'$ are isotopic in $F \times I$, then $c$ and $c'$ are 1-bridge isotopic. 
Hence $mbc_1(c)=bc_1(c)=bc_1(c')$.
\end{lem}

\begin{proof}
First, by shrinking $\beta$ and sliding the $\nu_i$ along with the endpoints of $\beta$, we may assume that $\beta$ is a very short arc in $F \times \set{\frac{2}{3}}$.  Observe that $\tau$ can be slid out of the way during this transformation, so that this operation is a 1-bridge isotopy. Let $p: F \times I \to F$ be the projection map defined by $p(x, t) = x$. Then, observe further that the number of double-points of $c$ under the map $p$ does not change during this transformation. Similarly, we may shrink $\beta'$, and then translate $\beta'$ through $F \times \set{\frac{2}{3}}$ via 1-bridge isotopy so that $\beta = \beta'$ (and therefore also so that $\nu_i = \nu_i'$ for $i=1,2$). 

Let $s = (x_1, \frac{1}{3})$, one of the endpoints of $\tau$. The projection map induces an isomorphism on fundamental groups, so that $\pi_1(F \times I, s) \cong \pi_1(F, x_1)$ via $p_*$. Then, since $c$ and $c'$ are isotopic in $F \times I$, we have $[c]_{F \times I} = \ell^{-1} * [c']_{F \times I} * \ell$ for some word $\ell \in \pi_1(F \times I, s)$. In fact, up to homotopy in $\pi_1(F \times I, s)$, we can take $\ell$ to be a loop in $F \times \set{\frac{1}{3}}$, based at $s$, containing the arc parallel to $\beta$ in $F \times \set{\frac{1}{3}}$ as a sub-arc, and never intersecting the arc parallel to $\beta$ in $F \times \set{\frac{1}{3}}$.

We now perform a 1-bridge isotopy of $c'$ by dragging $\beta' \cup \nu_1' \cup \nu_2'$ along $\ell$. Any time $\ell$ intersects $\tau'$, move $\tau'$ out of the way of the feet of $\nu_1' \cup \nu_2'$, dragging $\tau'$ along for the duration of the isotopy. Any time $\ell$ intersects itself, the isotopy will eventually run into $\tau'$ a second time, so we simply drag it along in the same way, see Figure \ref{fig:1bridgeisotopy}. Call the result $c'' = \tau'' \cup \beta'' \cup \nu_1'' \cup \nu_2''$. 
By design, we now have $[c'']_{F \times I} =  \ell^{-1} * [c']_{F \times I} * \ell= [c]_{F \times I}$. Thus, $[p(c)]_F = p_*([c]_{F \times I}) = p_*([c'']_{F \times I}) = [p(c'')]_F$. 
\begin{figure}[h]
\begin{center}
\includegraphics[width=13cm]{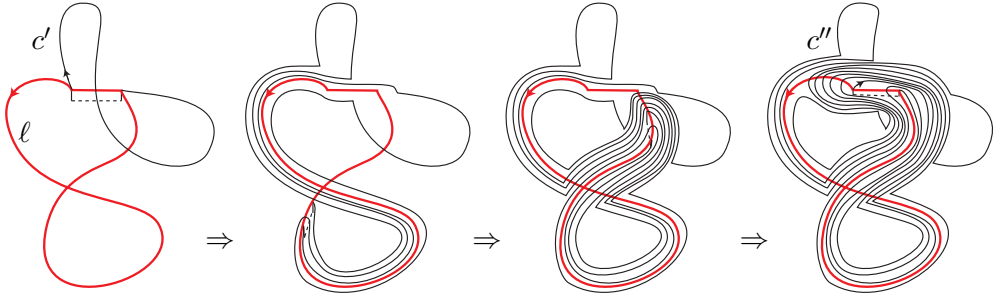}

\begin{picture}(400,0)(0,0)
\put(90,30){$\Rightarrow$}
\put(190,30){$\Rightarrow$}
\put(290,30){$\Rightarrow$}
\put(20,70){$\ell$}
\put(25,105){$c'$}
\put(315,105){$c''$}
\end{picture}
\caption{1-bridge isotopy along $\ell$.}
\label{fig:1bridgeisotopy}
\end{center}
\end{figure}
Now, since $\beta = \beta''$, we in fact know that $p(\tau)$ and $p(\tau'')$ are homotopic in $F$, and hence are isotopic in $F$, fixing endpoints (see \cite{BaeIKOGFZTDF}). This isotopy clearly lifts to a 1-bridge isotopy from $c$ to $c''$.  
Thus, ultimately $c$, $c'$ and $c''$ are all related by 1-bridge isotopy. 
\end{proof}

\begin{lem}
\label{lem:1BridgeFiberedLink}
Let $c$ and $c'$ be loops in the complement of a fibered link, $L$, in 1-bridge positions. If $c$ and $c'$ are isotopic in $M \rmv L$, then there exists an integer $k$, so that $c$ and $H^k(c')$ are 1-bridge isotopic, where $H$ is the natural automorphism of $M \rmv L$ induced by the monodromy of the link. Hence  $mbc_1(c)=bc_1(c)=bc_1(c')$.
\end{lem}

\begin{proof}
Consider the infinite cyclic cover $\wt{N} \cong F \times \mathbb{R}$, with covering map $P: \wt{N} \to N = M \rmv L$ defined by $P(x, t) =  (h^{k}(x), t - k)$, for $t \in [k, k+1]$. Then, the automorphism $H$ lifts to a map $T: \wt{N} \to \wt{N}$ defined by $T(x, t) = (x, t+1)$.

The isotopy from $c$ to $c'$ in $M \rmv L$ lifts to an isotopy from a lift $\wt{c}$ of $c$ to a lift $\wt{c'}$ of $c'$, in $\wt{N}$. By relabeling if necessary, we may take $\wt{c'}$ to be in $F \times [0,1] \subset F \times \mathbb{R}$, and $\wt{c}$ to be in $F \times [k, k+1]$ for some $k \in \mathbb{Z}_{\geq 0}$. Then let $\wh{c'} = T^{k}(\wt{c'})$, so that $\wh{c'}$ is isotopic to $\wt{c}$ in $\wt{N}$, and $\wh{c'} \subset F \times [k, k+1]$.

The isotopy from $\wh{c'}$ to $\wt{c}$ is supported in a compact region of $\wt{N}$, so we can restrict our attention to $F \times [m, n]$, for some $m,n \in \mathbb{Z}$, with $m \leq 0 < k + 1 \leq n$.

Now, $\wh{c'}$ and $\wt{c}$ can be considered to be in 1-bridge positions in $F \times [m,n]$. Hence, by Lemma \ref{lem:1BridgeFxI}
, $\wh{c'}$ and $\wt{c}$ are 1-bridge isotopic in $F \times [m, n] \subset \wt{N}$. This descends to a 1-bridge isotopy in $M\rmv L$ from $H^{k}(c')$ to $c$. 

\end{proof}

Let $F$ be a fiber surface of a fibered link, and let $\alpha$ be a properly embedded arc in $F$. Recall that a loop formed by `pushing-off' $\alpha$ from $F$ is an \emph{$\alpha$-loop}. 

Now, we will use these results to characterize arc-loops in a surface which can be pushed into the surface with all crossings  
contained in a single sub-arc.

\begin{lem}
\label{lem:1BridgeArcLoop} 
If $c$ 
is an $\alpha$-loop, then $\rho(\alpha) = mbc_1(c)$.
\end{lem}

\begin{rem} \label{rem:1BridgeArcLoop} Observe that every $\alpha$-loop is in a 1-bridge position with $\rho(\alpha)$ crossings,  
but there probably exist loops with 1-bridge positions that are not isotopic to arc-loops.
\end{rem}

\begin{proof}
Let $c'$ be an arc-loop representative isotopic to $c$. Then, as in Remark \ref{rem:1BridgeArcLoop}, $c'$ is in 1-bridge position. Hence, by Lemma \ref{lem:1BridgeFiberedLink}
, there is a 1-bridge isotopy taking $c$ to $c'$, and $bc_1(c')=mbc_1(c)$.  
Further, $p(c')$ is exactly the result of pushing half of $c'$ through the monodromy, so that $bc_1(c') = \rho(\alpha)$.
\end{proof}

Now we will characterize Dehn surgeries along arc-loops 
preserving a fiber surface.
Since any arc-loop bounds a disk in $M$, we can take the preferred longitude for surgery slopes.
\medskip

\noindent
{\bf Theorem \ref{thm:surgery}.}
{\it
Suppose $F$ is a fiber surface in $M$ and $c$ is an $\alpha$-loop. 
Suppose that $\gamma$ is a non-trivial slope on $c$, and that $N(\gamma)$ is the manifold obtained from $M$ via the $\gamma$-surgery on $c$. 
Then $F$ is a fiber surface in $N(\gamma)$ if and only if 
\begin{enumerate}[\rm (1)]
\item $\alpha$ is clean and $\gamma=i_{\partial}(\alpha)+\frac{1}{n}$  for some integer $n$, or
\item $\alpha$ is once-unclean and $\gamma=i_{\partial}(\alpha)$.
\end{enumerate}
}
\medskip

\begin{proof}
Since $F$ is fiber surface and $c$ is disjoint from $F$, 
we may assume $c$ is in $\overline{M\rmv n(F)}=F\times I$.  
Theorem \ref{thm:Ni} tells us that $mbc_1(c)\le1$  
in the case when $F$ is a fiber surface in $N(\gamma)$. Then, by Lemma \ref{lem:1BridgeArcLoop}, $\alpha$ is either clean, or once-unclean with respect to the monodromy for $F$, depending on whether $c$ has zero or one crossings, respectively. 

Let $\lambda$ be the preferred longitude on $c$, $\mu$ be a meridian for $c$, and $\lambda_b$ the `black-board' frame induced by the surface $F$ together with the small bridge, as in \cite{NiDSKPM}. Then $\lambda_b = \lambda + i_{\partial}(\alpha) \cdot \mu$. 
By Theorem \ref{thm:Ni},  
the surgery slope $\gamma$ must be $$n \cdot \lambda_b + \mu = n \cdot \lambda + (n \cdot i_{\partial}(\alpha) + 1) \cdot \mu$$ if $c$ has no crossing,  
and must be $$\lambda_b = \lambda + i_{\partial}(\alpha) \cdot \mu$$ if $c$ has a single crossing. 
\end{proof}

%% file: crossing.tex
In this section we will characterize generalized crossing changes between fibered links.
Throughout this section, $L$ and $L'$ are oriented links in a manifold $M$  
related by a \emph{generalized crossing change}.
More precisely,  there exists a disk $D$ in $M$ 
such that $L$ intersects $D$ in two points with opposite orientations, 
and $L'$ is the image of $L$ after $(-\frac{1}{n})$-Dehn surgery along $c=\partial D$ for some $n \in \mathbb{Z} \rmv \set{0}$. The curve $c$ is called a \emph{crossing circle}, and we say that $L'$ is the result of a generalized crossing change \emph{of order $n$}.
When $n = \pm 1$, this is just an ordinary crossing change.
  
Scharlemann and Thompson \cite{SchaThoLGCM} showed that in the case $n= \pm1$, there exists a taut Seifert surface $F$ for $L$ or $L'$, say $L$, such that $F$ is disjoint from $c$ 
but intersects $D$ in an arc, and described surface locally. For $|n| > 1$, Kalfagianni and Lin \cite{KalLinKAGET} showed a similar result.

\begin{thm}[\cite{SchaThoLGCM}, \cite{KalLinKAGET}]\label{thm:Eulerc}
Suppose $L'$ is obtained from $L$ in $S^3$  
by a generalized crossing change of order $n$. Then $\chi(L')\ge\chi(L)$ if and only if $L$ has a taut Seifert surface $F$ 
such that $F$ is disjoint from $c$ 
but intersects $D$ in an arc. 
Moreover, $\chi(L')>\chi(L)$ if and only if $F$ is a plumbing of a $(-n)$-times twisted annulus, $A$, and 
a surface, $F''$, which is disjoint from $D$, and the result, $A'$, of $A$ after the twist is compressible.
\end{thm}

Suppose now that $L$ is a fibered link. 
Let $\alpha$ be the arc $D\cap F$ in Theorem \ref{thm:Eulerc}.
Recall that we call $c$ an $\alpha$-loop.
We will say that performing the generalized crossing change ($(-\frac{1}{n})$-Dehn surgery) along $c$ is an  $n$-\emph{twist along $\alpha$}.
Here an $\varepsilon$-twist is right- or
left-handed if
$\varepsilon = 1$ or $-1$, respectively.

As mentioned earlier, the result of a plumbing of two surfaces is a fiber surface if and only if both summands are fiber surfaces (\cite{GabMSNGO,GabDFLS3}). Further, the only fiber annuli are the left- and right-handed Hopf annuli. Thus, by Theorem \ref{thm:Hopfband}, we can restate the last part of Theorem \ref{thm:Eulerc} as follows: 
\begin{thm}\label{thm:Eulerc2}
Suppose $L$ is a fibered link in $S^3$  
with fiber $F$, and
$L'$ is obtained from $L$ by an $n$-twist along $\alpha$, 
where $\alpha$ is a properly embedded arc in $F$.
Then $\chi(L')>\chi(L)$ if and only if $n = \pm 1$, and $\alpha$ is clean and alternating with $i_{\bd}(\alpha) = -n$.
\end{thm}
Moreover, Kobayashi  
showed that the resulting surface of the $n$-twist along $\alpha$ is a pre-fiber surface \cite[Theorem 2]{KobFLUO}, and he also characterized $\alpha$ in the pre-fiber surface  \cite[Lemma 4.7 and Proposition 8.1]{KobFLUO}.  
For the remaining case, 
we will characterize generalized crossing changes between fibered links $L$ and $L'$  with $\chi(L)=\chi(L')$.

Observe that if a crossing circle is nugatory (i.e. bounds a disk in the complement of the link), then any generalized crossing change will not change the link. For the case of knots, Kalfagianni \cite{KalCCCFK} showed the converse holds: if a crossing change on a fibered knot yields a fibered knot that is isotopic to the original, then the crossing circle must be nugatory. 

Stallings proved if $F$ is a fiber surface, and the loop $c$ is isotopic into $F$ so that the framing on $c$ induced by $F$ agrees with that of $D$, then the image of $F$ after $\pm 1$-Dehn surgery along $c$ is a fiber surface for the resulting link \cite{StaCFKL}. This came to be known as a \emph{Stallings twist}. Yamamoto proved that twisting along an arc is a Stallings twist if and only if the arc $\alpha$ is clean and non-alternating \cite{YamSTWCBRPDHB} (see also Theorem \ref{thm:4equiv}). (Note that a crossing change is nugatory if and only if the arc $\alpha$ is fixed by the monodromy. In this case, $\alpha$ is clean and non-alternating. Since the crossing circle can be isotoped to a trivial loop in the surface $F$, this can also be considered a special case of a Stallings twist.)

We generalize Yamamoto's result and characterize exactly when twisting along an arc results in a fiber surface.
\bigskip

\begin{thm} \label{thm:twistsurface}
Suppose 
$F$ is a fiber surface, 
and $\alpha$ is a properly embedded arc in $F$.
Let $F'$ be the resulting surface of an $\varepsilon$-twist along $\alpha$ for $\varepsilon \in \set{\pm 1}$.
Then $F'$ is a 
fiber surface 
if and only if $i_{total}(\alpha) = 0$ (i.e., $\alpha$ is clean and non-alternating)
or $\alpha$ is once-unclean and alternating with $i_{\bd}(\alpha) = -\varepsilon$.
\end{thm}

\begin{proof} Plumb a Hopf annulus along an arc parallel to the boundary of $F$, with endpoints on either side of $\alpha$, 
so that the trivial sub-disk cut off by this arc contains only one point, $p$, of $\bd \alpha$. 
The result is a new 
fibered link, 
together with its fiber, $F''$.
Observe that the monodromy of $F''$ differs from that of $F$ by exactly a Dehn twist along the core of the newly plumbed on Hopf annulus, right- or left-handed depending on the twist of the Hopf annulus.

Now, the result of cutting $F''$ along $\alpha$ is exactly $F'$. So, by Theorem \ref{thm:whenfiber}, $F'$ is a 
fiber
if and only if $\alpha$ is clean, alternating, or once-unclean, non-alternating in $F''$. The arc $\alpha$ will be clean, alternating in $F''$ exactly when $\alpha$ is clean, non-alternating in $F$ and the sign of the Hopf annulus disagrees with the sign of $i(\alpha, h(\alpha))$ at $p$ in $F$, where $h$ is the monodromy of $F$. The arc $\alpha$ will be once-unclean, non-alternating in $F''$ exactly when either $\alpha$ is clean, non-alternating in $F$ and the sign of the Hopf annulus agrees with the sign of $i(\alpha, h(\alpha))$ at $p$ in $F$, or when $\alpha$ is once-unclean, alternating in $F$, and the sign of the Hopf annulus disagrees with the sign of $i(\alpha, h(\alpha))$ at $p$ in $F$. 
\end{proof}

By Theorem \ref{thm:Eulerc}, we have Theorem \ref{thm:crossingchange} as a translation of Theorem \ref{thm:twistsurface}.

\medskip

\noindent
{\bf Theorem \ref{thm:crossingchange}.}
{\it
 Suppose a link $L'$ is obtained from a fibered link $L$ in $S^3$ 
with fiber $F$ by a crossing change, and $\chi(L')=\chi(L)$. 
Then $L'$ is a fibered link if and only if the crossing change is a Stallings twist or an $\varepsilon$-twist along an arc $\alpha$ in $F$, 
where $\alpha$ is once-unclean and alternating with $i_{\bd}(\alpha)=-\varepsilon$.}
\medskip

In fact, using Theorem \ref{thm:surgery},  
we can characterize any generalized crossing change between fibered links of the same Euler characteristic.

\medskip

\noindent
{\bf Theorem \ref{thm:gencrossingchange}.}
{\it
Suppose $L$ and $L'$ are fibered links in $S^3$ 
related by a generalized crossing change with $\chi(L) = \chi(L')$.
Then the generalized crossing change is an $n$-twist around an arc $\alpha$, and one of the following holds:
\begin{enumerate}[\rm (1)]
\item $\alpha$ is clean and non-alternating, 
\item $n=\pm2$, and $\alpha$ is clean and alternating with $i_{\partial}(\alpha)=-n/2$, or
\item $n=\pm1$, and $\alpha$ is once-unclean and alternating with $i_{\partial}(\alpha)=-n$.
\end{enumerate}}
 
 \medskip

\begin{proof}
Let $F$ be a fiber surface of $L$. 
By Theorem \ref{thm:Eulerc}, $c$ is an $\alpha$-loop for some arc $\alpha$ in $F$ and $F$ is a fiber surface after $(-\frac{1}{n})$-surgery on $c$. 
Then, by Theorem \ref{thm:surgery}, $\alpha$ is clean and
$$-\frac{1}{n}=i_{\partial}(\alpha)+\frac{1}{m}$$ for some integer $m$, or 
$\alpha$ is once-unclean and $$-\frac{1}{n}=i_{\partial}(\alpha).$$
If $\alpha$ is clean, then either $i_{\partial}(\alpha)=0$ so $\alpha$ is non-alternating, or $n=\pm2$ and $i_{\partial}(\alpha)=-n/2$.   
If $\alpha$ is once-unclean, then $n=\pm1$ and $i_{\partial}(\alpha)=-n$. 
\end{proof}

\begin{cor}\label{cor:gencrossingchange-monodromy}
If $L$ and $L'$ are related as above, and the twist is around a once-unclean, alternating arc, then the monodromy map changes by composition with $t_a^2 t_b^2 t_c^{-1}$ or $t_a^{-2} t_b^{-2} t_c$, depending not on $n$, but on $i_p(\alpha, h(\alpha))$ at the interior point of intersection between $\alpha$ and $h(\alpha)$, where $t_a$ denotes a Dehn twist about the curve $a$, and $a, b, c$ are the loops formed by resolving the intersection of $\alpha \cup h(\alpha)$ in two ways, as in \cite{NiDSKPM}.
\end{cor}

\begin{proof}
This follows from Proposition 1.4 of \cite{NiDSKPM}.
\end{proof}

%% file: alternativeproof.tex
In this section we will give an alternative proof of Theorem \ref{thm:whenfiber} using Theorem \ref{thm:surgery}.   
Let $L$ be a fibered link in a manifold $M$ with fiber $F$, and let $F'$ be a surface obtained from $F$ by cutting along an arc $\alpha$. 
Let $c$ be an $\alpha$-loop. 
We consider $F$ in $N(0)$ which is obtained from $M$ by $0$-surgery on $c$. 
Theorem \ref{thm:surgery} gives the following necessary and sufficient condition for $F$ to be a fiber surface in $N(0)$.
\begin{enumerate}
\item $\alpha$ is clean and $i_{\partial}(\alpha)=\pm1$, or
\item $\alpha$ is  once-unclean and 
$i_{\partial}(\alpha)=0$. 
\end{enumerate}
This is the same condition as in Theorem \ref{thm:whenfiber}. 
Then Theorem \ref{thm:whenfiber} follows from
Lemma \ref{lem:FandF'} below.
\begin{lem}\label{lem:FandF'}
$F'$ is a fiber surface in $M$ if and only if $F$ is a fiber surface in $N(0)$.
\end{lem}
\begin{rem} A statement analogous to Lemma \ref{lem:FandF'} holds replacing fiber surface with taut surface, since tautness is also invariant under product decomposition and reverse operations.
\end{rem}
\begin{proof}
The idea of this proof is based on {\it Proof of Claim 2} in \cite{SchaThoLGCM}.
Take a small neighborhood $n(c)$ of $c$ and a small product neighborhood $n(F)=F\times I$ of $F$ so that $n(c)$ and $n(F)$ are disjoint.
Let $D$ be a disk $\alpha\times I$ in $n(F)$ and let $\beta$ be a loop $\partial D$ in $\partial (n(F))$. 
By the definition of $\alpha$-loop, there exists an annulus in $\overline{M\rmv (n(c)\cup n(F))}$ with boundary components $\beta$ and a longitude $\lambda$ on $\partial (n(c))$. 
Then $\beta$ bounds a disk $D'$ in $\overline{N(0)\rmv n(F)}$, the union of the annulus and a meridional disk of the solid torus filled into $N(0)$. 
Since $\beta$ intersects the suture $\partial F$ at two points, $D'$ is a product disk for the sutured manifold $(\overline{N(0)\rmv n(F)},\partial F)$.
A product neighborhood  $n(F')=F'\times I$ is obtained from $n(F)$ by removing a neighborhood of $D$, and so $\overline{M\rmv n(F')}$ is obtained from $\overline{M\rmv n(F)}$ by attaching a $2$-handle along $\beta$. 
Attaching to $\overline{M\rmv n(F)}$ a $2$-handle along $\beta$ is equivalent to deleting from $\overline{N(0)\rmv n(F)}$ a neighborhood of $D'$. 
Then a sutured manifold $(\overline{M\rmv n(F')},\partial F')$ is obtained from $(\overline{N(0)\rmv n(F)},\partial F)$ by decomposing along $D'$.  
$$ (\overline{N(0)\rmv n(F)}, \partial F) \stackrel{D'}{\leadsto} (\overline{M\rmv n(F')}, \partial F').$$  
Since the triviality of a sutured manifold is invariant under product decomposition and reverse operations, 
$F'$ is a fiber surface in $M$ if and only if $F$ is a fiber surface in $N(0)$.
\end{proof}

\begin{cor}\label{cor:monodromy-band}
Suppose $F$ and $F'$ are related as above, and $\alpha$ is once-unclean, non-alternating. 
Let $h$ and $h'$ be monodromies of $F$ and $F'$ respectively.
Then $(t_a^2 t_b^2 t_c^{-1}  h)|_{F'}=h'$ or $(t_a^{-2} t_b^{-2} t_c  h)|_{F'}=h'$, depending on $i_p(\alpha, h(\alpha))$ at the interior point of intersection between $\alpha$ and $h(\alpha)$, where $t_a$ denotes a Dehn twist about the curve $a$, and $a, b, c$ are the loops formed by resolving the intersection of $\alpha \cup h(\alpha)$ in two ways, as in \cite{NiDSKPM}.
\end{cor}
\begin{rem}
Theorem \ref{thm:gHopfband} tells us that $F$ is a generalized Hopf banding of $F'$. 
The loops $a,b,c$ in Corollary \ref{cor:monodromy-band} for a generalized Hopf banding are depicted in Figure \ref{fig:monodromygHopf}.
For an arc $\ell$ in $F'$ with a single self-intersection point, there are two generalized Hopf bandings depending on which part of the band is in the higher position at the place of overlap. 
The two monodromies in Corollary \ref{cor:monodromy-band} correspond to these two surfaces.  
If the self-intersection point of $\ell$ is removable in $F'$, then the genralized Hopf banding is a Hopf banding. 
In that case,  $b$ is trivial in $F'$ (and so in $F$), $a$ and $c$ are isotopic to each other in $F$, 
and so $t_a^2 t_b^2 t_c^{-1}=t_a, t_a^{-2} t_b^{-2} t_c =t_a^{-1}$, is a Dehn twist along the core of the Hopf annulus.
\end{rem}
\begin{figure}[h]
\begin{center}
\includegraphics[width=6cm]{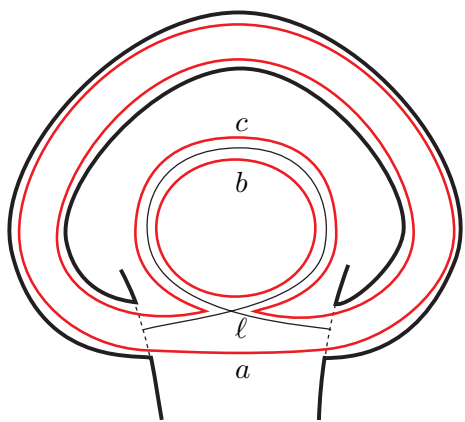}

\begin{picture}(400,0)(0,0)
\put(200,45){$\ell$}
\put(200,123){$c$}
\put(200,100){$b$}
\put(200,30){$a$}
\end{picture}
\caption{Loops of Dehn twists for a generalized Hopf banding.}
\label{fig:monodromygHopf}
\end{center}
\end{figure}
\begin{proof}
Let $h_0$ be the monodromy of $F$ in $N(0)$, i.e. $(F\times [0,1]) / h_0$ is homeomorphic to $\overline {N(0)\rmv n(F)}$. 
By Proposition 1.4 of \cite{NiDSKPM}, then, $t_a^2 t_b^2 t_c^{-1}  h=h_0$ or $t_a^{-2} t_b^{-2} t_c  h=h_0$. 
In the proof of Lemma \ref{lem:FandF'}, recall that the product disk $D'$ of $(\overline{N(0)\rmv n(F)},\partial F)$ is the boundary of a disk $\alpha\times I$, and $(\overline{M\rmv n(F')},\partial F')$ is obtained from $(\overline{N(0)\rmv n(F)},\partial F)$ by decomposing along $D'$. 
This implies that $h_0(\alpha)=\alpha$ and $h_0|_{F'}=h'$.
\end{proof}